\documentclass{amsart}
\usepackage[latin1]{inputenc}
\usepackage{amssymb}
\usepackage{amsmath}
\usepackage{amsfonts}
\usepackage{amsthm}
\usepackage{enumerate}
\usepackage{color}

\numberwithin{equation}{section}

\newtheorem{theorem}{Theorem}[section]
\newtheorem{lemma}{Lemma}[section]
\newtheorem{proposition}{Proposition}[section]
\newtheorem{corollary}{Corollary}[section]

\newtheorem{remark}{Remark}[section]

\def\R{\mathbb{R}}
\def\Z{\mathbb{Z}}

\def\D{\mathbb{D}}
\def\S{\mathcal{S}}

\def\N{\mathcal{N}}

\def\F{\mathcal{F}}

\def\eps{\varepsilon}

\def\supp{\, \mathop{\rm supp}\nolimits}
\def\sgn{\mathop{\rm sgn}\nolimits}
\newcommand\cro[1]{\langle #1 \rangle}

\begin{document}

\title[Well-posedness for BOZK]{Local and global well-posedness results for the Benjamin-Ono-Zakharov-Kuznetsov equation}
\author[F. Ribaud]{Francis Ribaud}
\address{Universit\'e Paris-Est Marne-la-Vall\'ee\\ Laboratoire d'Analyse et de Math\'ematiques Appliqu\'ees (UMR 8050)\\  5 Bd Descartes, Champs-sur-Marne\\ 77454 Marne-la-Vall\'ee Cedex 2
\\ France }
\email{francis.ribaud@univ-mlv.fr}
\author[S. Vento]{St\'ephane Vento}
\address{Universit\'e Paris 13\\ Sorbonne Paris Cit\'e, LAGA, CNRS (UMR 7539)\\  99, avenue Jean-Baptiste Cl\'ement\\ F-93 430 Villetaneuse\\ France }
\email{vento@math.univ-paris13.fr}
\maketitle

\begin{abstract}
We show that the initial value problem associated to the dispersive generalized Benjamin-Ono-Zakharov-Kuznetsov equation
$$
 u_t-D_x^\alpha u_{x} + u_{xyy} = uu_x,\quad (t,x,y)\in\R^3,\quad 1\le \alpha\le 2,
$$
is locally well-posed in the spaces $E^s$, $s>\frac 2\alpha-\frac 34$, endowed with the norm
$
\|f\|_{E^s} = \|\cro{|\xi|^\alpha+\mu^2}^s\hat{f}\|_{L^2(\R^2)}.
$
As a consequence, we get the global well-posedness in the energy space $E^{1/2}$ as soon as $\alpha>\frac 85$. The proof is based on the approach of the short time Bourgain spaces developed by Ionescu, Kenig and Tataru \cite{IKT} combined with new Strichartz estimates and a modified energy.
\end{abstract}

\section{Introduction}
In this paper we study a class of two-dimensional nonlinear dispersive equations which extend the well-known Korteweg-de Vries (KdV) and Benjamin-Ono (BO) equations. There are several ways to generalize such 1D models in order to include the effect of long wave lateral dispersion. For instance one can consider the Kadomstev-Petviashvili (KP) and Zakharov-Kuznetsov (ZK) equations.
Here we are interested with the effect of the dispersion in the propagation direction applied to the initial value problem for the ZK equation. More precisely we consider the generalized g-BOZK equation
\begin{equation}\label{bozk}
 u_t-D_x^\alpha u_{x} + u_{xyy} = uu_x,\quad (t,x,y)\in\R^3
\end{equation}
where $D_x^\alpha$ is the Fourier multiplier by $|\xi|^\alpha$, $1\le\alpha\le 2$. When $\alpha=2$, \eqref{bozk} is the well-known ZK equation introduced by Zakharov and Kuznetsov in \cite{ZK} to describe the propagation of ionic-acoustic waves in magnetized plasma. We refer to \cite{LLS} for a rigorous derivation of ZK. For $\alpha=1$, equation \eqref{bozk} is the so-called Benjamin-Ono-Zakharov-Kuznetsov (BOZK) equation introduced in \cite{JCMS} and \cite{LMSV} and has applications to thin nanoconductors on a dielectric substrate.

We notice that  \eqref{bozk} enjoys the two following conservation laws:
\begin{equation}\label{conslaw}
\frac{d}{dt}\mathcal{M}(u) = \frac{d}{dt}\mathcal{H}(u) = 0,
\end{equation}
where
$$
\mathcal{M}(u) = \int_{\R^2} u^2 dxdy
$$
and
$$
\mathcal{H}(u) = \int_{\R^2} \left( |D^{\frac \alpha 2}_xu|^2 + |u_y|^2 - \frac 13 u^3 \right)dxdy.
$$
Therefore, it is natural to study the well-posedness of g-BOZK in the functional spaces $E^0$ and $E^{1/2}$, and more generally in $E^s$ defined for any $s\in\R$ by the norm
$$
\|f\|_{E^s} = \|\cro{|\xi|^\alpha+\mu^2}^s\hat{f}(\xi,\mu)\|_{L^2(\R^2)}.
$$
Observe that $E^s$ is nothing but the anisotropic Sobolev space $H^{\alpha s, 2s}(\R^2)$. In particular when $\alpha=2$, then $E^s=H^{2s}(\R^2)$.

Let us recall some well-known facts concerning the associated 1D model
\begin{equation}\label{gBO}
u_t-D^\alpha_xu_x = uu_x,\quad (t,x)\in\R^2.
\end{equation}
The Cauchy problem for \eqref{gBO}, and especially the cases $\alpha=1,2$ (respectively the BO and KdV equation), has been extensively studied these last decades, and is now well-understood.
The standard fixed point argument in suitable functional spaces allows to solve the KdV equation at very low regularity level (see \cite{KPV} for instance). This is in sharp contrast with what occurs in the case $\alpha<2$, since it was shown by Molinet-Saut-Tzvetkov \cite{MST} that the solution flow map for \eqref{gBO} cannot be $C^2$ in any Sobolev spaces (due to bad low-high interactions). Therefore the problem cannot be solved using such arguments. In view of this result, three approaches were developed to lower the regularity requirement. The first one consists in introducing a nonlinear gauge transform of the solution that solves an equation with better interactions (see \cite{T}-\cite{HIKK}). This method was proved to be very efficient but as pointed out in \cite{CP},  it is not clear how to find such a transform adapted to our 2D problem \eqref{bozk}. The second one was introduced very recently by Molinet and the second author \cite{MV} and consists in an improvement of the classical energy method by taking into account the dispersive effect of the equation. This method is more flexible with respect to perturbations of the equation but requires that the dispersive part of the equation does not exhibit too strong resonances. Unfortunately, the cancelation zone of the resonance function $\Omega$ associated to g-BOZK (see \eqref{omega} for the definition) seems too large to apply this technique to equation \eqref{bozk}. Finally the third method introduced to solve \eqref{gBO} consists in improving dispersive estimates by localizing it in space frequency depending time intervals. In the context of the Bourgain spaces, this approach was successfully applied by Guo in \cite{G} to solve \eqref{gBO} (see also \cite{IKT} for an application to the KP-I equation) and seems to be the best way to deal with the g-BOZK equation.

Now we come back to the 2D problem \eqref{bozk}. The initial value problem for the ZK equation ($\alpha=2$) has given rise to many papers these last years. In particular, Faminskii proved in \cite{F} that it is globally well-posed in the energy space $H^1(\R^2)$. The best result concerning the local well-posedness was recently independently obtained by Gr\"unrock and Herr in \cite{GH} and by Molinet and Pilod in \cite{MP} where they show the LWP of \eqref{bozk} in $H^s(\R^2)$, $s>1/2$. Similarly to the KdV equation, all these results were proved using the fixed point procedure. Concerning the case $\alpha=1$, using classical energy methods and parabolic regularization that does not take into account the dispersive effect of the equation, Cunha and Pastor \cite{CP} have proved the well-posedness of \eqref{bozk} in $H^s(\R^2)$ for $s>2$ as well as in the anisotropic Sobolev spaces $H^{s_1,s_2}(\R^2)$, $s_2>2$, $s_1\ge s_2$. Also, it was proved in \cite{EP} that the solution mapping fails to be $C^2$ smooth in any $H^{s_1,s_2}(\R^2)$, $s_1,s_2\in \R$. Moreover this result even extends to the case $1\le\alpha <\frac 43$.

In the intermediate cases $1<\alpha<2$, there is no positive results concerning the well-posedness for \eqref{bozk}. Our main theorem is the following.

\begin{theorem}\label{theo-main}
  Assume that $1\le \alpha\le 2$ and $s>s_\alpha:= \frac 2\alpha-\frac 34$. Then for every $u_0\in E^s$, there exists a positive time $T=T(\|u_0\|_{E^s})$ and a unique solution $u$ to \eqref{bozk} in the class
  $$
  C([-T, T]; E^s)\cap F^s(T)\cap B^s(T).
  $$
  Moreover, for any $0<T'<T$, there exists a neighbourhood $\mathcal{U}$ of $u_0$ in $E^s$ such that the flow map data-solution
  $$
  S_{T'}^s : \mathcal{U}\to C([-T',T']; E^s), u_0\mapsto u,
  $$
  is continuous.
\end{theorem}

\begin{remark} We refer to Section \ref{spaces} for the definition of the functional spaces $F^s(T)$ and $B^s(T)$.
\end{remark}

\begin{remark}
  When $\alpha=2$, we recover the local well-posedness result in $E^{1/4+}=H^{1/2+}(\R^2)$ for ZK proved in \cite{GH} and \cite{MP}. In the case $\alpha=1$, Theorem \ref{theo-main} improves the previous results obtained in \cite{CP}.
\end{remark}

We discuss now some of the ingredients in the proof of Theorem \ref{theo-main}. We will adapt the approach introduced by Ionescu, Kenig and Tataru \cite{IKT} to our model (see also \cite{G}-\cite{KP} for applications to other equations). It consists in an energy method combined with linear and nonlinear estimates in the short-time Bourgain's spaces $F^s(T)$ and their dual $\N^s(T)$. The $F^s(T)$ spaces enjoys a $X^{s,1/2}$-type structure but with a localization in small time intervals whose length is of order $H^{1-\frac 2\alpha}$ when the space frequency $(\xi,\mu)$ satisfies $|\xi|^\alpha+\mu^2 \sim H$. When deriving bilinear estimates in these spaces, one of the main obstruction is the strong resonance induced by the dispersive part of the equation. To overcome this difficulty, we will derive some improved Strichartz estimates for free solutions localized outside the critical region $\{2\mu^2= \alpha(\alpha+1)|\xi|^\alpha\}$. Finally, we need energy estimates in order to apply the classical Bona-Smith argument (see \cite{BS}) and conclude the proof of Theorem \ref{theo-main}. To derive such energy estimates, we are led to deal with terms of the form
$$
\int_{\R^2} P_Hu P_H(uu_x),
$$
where $P_H$ localizes in the frequencies $\{|\xi|^\alpha+\mu^2\sim H\}$.
Unfortunately, in the two-dimensional setting, we cannot put the $x$-derivative on the lower frequency term \emph{via} commutators and integrations by parts without loosing a $y$-derivative. Therefore, we need to add a cubic lower-order term to the energy in order to cancel those bad interactions.

Assuming that $s_\alpha < \frac 12$, we may use the conservation laws \eqref{conslaw} combined with the embedding $E^{1/2}\hookrightarrow L^3(\R^2)$ to get an a priori bound of the $E^{1/2}$-norm of the solution and then iterate Theorem \ref{theo-main} to obtain the following global well-posedness result.
\begin{corollary}\label{cor-main}
  Assume that $\frac 85 < \alpha\le 2$ and $s=\frac 12$. Then the results of Theorem \ref{theo-main} are true for $T>0$ arbitrary large.
\end{corollary}

Finally, as in the one dimensional case,  we show that as soon as $\alpha<2$, the solution map $S^s_T$ given by Theorem \ref{theo-main} is not of class $C^2$ for all $s\in\R$. This implies in particular that the Cauchy problem for \eqref{bozk} cannot be solved by direct contraction principle.

\begin{theorem}\label{ill-theo}
  Fix $s\in\R$ and $1\le\alpha<2$. Then there does not exist a $T>0$ such that \eqref{bozk} admits a unique local solution defined on the interval $[-T,T]$ and such that the flow-map data-solution $u_0\mapsto u(t)$, $t\in[-T,T]$ is $C^2$-differentiable at the origin from $E^s$ to $E^s$.
\end{theorem}

The rest of the paper is organized as follows: in Section \ref{sec-not}, we introduce the notations, define the function spaces and state some associated properties. In Section \ref{stri-sec}, we derive  Strichartz estimates for free solutions of \eqref{bozk}. In Section \ref{sec-l2bil} we show some $L^2$-bilinear estimates which are used to prove the main short time bilinear estimates in Section \ref{sec-bil} as well as the energy estimates in Section \ref{sec-energy}. Theorem \ref{theo-main} is proved in Section \ref{sec-main}. We conclude the paper with an appendix where we show the ill-posedness result of Theorem \ref{ill-theo}.

\section{Notations and functions spaces}\label{sec-not}
\subsection{Notations}\label{not}
For any positive numbers $a$ and $b$, the notation $a\lesssim b$ means that there exists a positive constant $c$ such that $a\le cb$. By $a\sim b$ we mean that $a\lesssim b$ and $b\lesssim a$.
Moreover, if $\gamma\in\R$, $\gamma+$, respectively $\gamma-$, will denote a number slightly greater, respectively lesser, than $\gamma$.

The Fourier variables of $(t,x,y)$ are denoted $(\tau,\xi,\mu)$. Let $U(t)=e^{t \partial_x (D^\alpha _x-\partial_{yy})}$ be the linear group associated with the free part of \eqref{bozk} and set
\begin{equation}\label{pomega}
\omega(\zeta) = \omega(\xi,\mu)=\xi(|\xi|^\alpha+\mu^2),
\end{equation}
\begin{equation}\label{omega}
\Omega(\zeta_1,\zeta_2) = \omega(\zeta_1+\zeta_2)-\omega(\zeta_1)-\omega(\zeta_2).
\end{equation}
Let $h$ the partial derivatives of $\omega$ with respect to $\xi$ :
$$
h(\xi,\mu) = \partial_\xi\omega(\xi,\mu) = (\alpha+1)|\xi|^\alpha+\mu^2,
$$
We define the set of dyadic numbers $\D = \{2^\ell, \ell\in\mathbb{N}\}$. If $\beta\ge 0$ and $H=2^\ell\in \D$, we will denote by $\lfloor H^\beta\rfloor$ the dyadic number such that $\lfloor H^\beta\rfloor\le H^\beta< 2\lfloor H^\beta\rfloor$. In other words we set $\lfloor H^\beta\rfloor = 2^{[\beta k]}$ where $[\cdot]$ is the integer part.

Let $\chi\in C^\infty_0$ satisfies $0\le \chi\le 1$, $\chi=1$ on $[-4/3, 4/3]$ and $\chi(\xi)=0$ for $|\xi|>5/3$.
Let $\varphi(\xi)=\chi(\xi)-\chi(2\xi)$ and for any $N\in\D\setminus\{1\}$, define $\varphi_N(\xi)=\varphi(\xi/N)$ and $\varphi_1=\chi$.
For $H,N\in\D$, we consider the Fourier multipliers $P_N^x$ and $P_H$ defined as
$$
\F(P_N^xu)(\tau,\xi,\mu) = \varphi_N(\xi)\F u(\tau,\xi,\mu),
$$
$$
\F(P_Hu)(\tau,\xi,\mu) = \psi_H(\xi,\mu)\F u(\tau,\xi,\mu),
$$
with $\psi_H(\xi,\mu)=\varphi_H(|\xi|^\alpha+\mu^2)$.

If $A\subset\R^2$, we denote by $P_A=\F^{-1}1_A\F$ the Fourier projection on $A$.

For $N,H\in \D\setminus\{1\}$, let us define
$$
I_N = \{\xi : \frac N2 \le |\xi|\le 2N\}, \quad I_1 = [-2,2],
$$
and
$$
\Delta_H=\{(\xi,\mu) : h(\xi,\mu)\in I_H\}\;.
$$
We also define $P_{\lesssim H}=\sum_{H_1\lesssim H}P_{H_1}$, $P_{\gg H}=Id-P_{\lesssim H}$ and $P_{\sim H}=Id-P_{\lesssim H}-P_{\gg H}$. We will use similarly the notation
$\varphi _{\leq}$, $\varphi_{\geq}$ ...

Let $\eta:\R^4\to\mathbb{C}$ be a bounded measurable function. We define the pseudo-product operator $\Pi_\eta$ on $\S(\R^2)^2$ by
$$
\F(\Pi_\eta(f,g))(\zeta) = \int_{\zeta=\zeta_1+\zeta_2} \eta(\zeta_1,\zeta_2)\widehat{f}(\zeta_1) \widehat{g}(\zeta_2).
$$
This bilinear operator enjoys the symmetry property
\begin{equation}\label{symPi}
  \int_{\R^2} \Pi_\eta(f,g)h = \int_{\R^2} f\ \Pi_{\eta_1}(g,h) = \int_{\R^2} \Pi_{\eta_2}(f,h)g
\end{equation}
with $\eta_1(\zeta_1,\zeta_2)= \overline{\eta}(\zeta_1+\zeta_2,-\zeta_2)$ and $\eta_2(\zeta_1,\zeta_2)=\overline{\eta}(\zeta_1+\zeta_2,-\zeta_1)$ for any real-valued functions $f,g,h\in\S(\R^2)$.
This operator behaves like a product in the sense that it satisfies
$$
\Pi_\eta(f,g) = fg \textrm{ if } \eta\equiv 1,
$$
\begin{equation}\label{PiPart}
\partial \Pi_\eta(f,g) = \Pi_\eta(\partial f, g) + \Pi_\eta(f,\partial g),
\end{equation}
for any $f,g\in \mathcal{S}(\R^2)$ where $\partial$ holds for $\partial_x$ or $\partial_y$. Moreover, if $f_i\in L^2(\R^2)$, $i=1,2,3$ are localized in $\Delta_{H_i}$ for some $H_i\in\D$, then
\begin{equation}\label{PiEst}
\left|\int_{\R^2} \Pi_\eta(f_1,f_2)f_3\right|\lesssim H_{min}^{\frac{1}{2\alpha}+\frac 14} \prod_{i=1}^3 \|f_i\|_{L^2(\R^2)}.
\end{equation}
Estimate \eqref{PiEst} follows from \eqref{symPi}, Plancherel's theorem and the fact that $\|\psi_H\|_{L^2(\R^2)}^2 \sim H^{\frac 1\alpha + \frac 12}$ for any $H\in \D$.

\subsection{Function spaces}\label{spaces}
If $\phi\in L^2(\R^3)$ is supported in $\R\times \Delta_H$ for $H\in\D$, the space $X_H$ is defined by the norm
$$
\|\phi\|_{X_H}  = \sum_{L\in\D} L^{1/2}\|\varphi_L(\tau-\omega(\zeta))\phi(\tau,\zeta)\|_{L^2_{\tau\zeta}}.
$$
For a function $f\in L^2(\R^3)$ such that $\F(f)$ is supported in $\R\times \Delta_H$ for $H\in\D$, we introduce the Bourgain's space $F_H$ localized in short time intervals of length $H^{-\beta}$ where $\beta$ is fixed to
$$
\beta = \frac 2\alpha-1\ge 0,
$$
defined by the norm
\begin{equation}\label{def.FN}
\|f\|_{F_H} = \sup_{t_H\in\R} \|\F(\varphi_1(H^{\beta}(\cdot-t_H))f)\|_{X_H}.
\end{equation}
Its dual version $\N_H$ is defined by the norm
\begin{equation}\label{def.NH}
\|f\|_{\N_H} = \sup_{t_H\in\R} \| (\tau-\omega(\zeta)+iH^{\beta})^{-1}\cdot \F(\varphi_1(H^{\beta}(\cdot-t_H))f)\|_{X_H}.
\end{equation}
Now if $s\ge 0$, we define the global $F^s$ and $\N^s$ spaces from their frequency localized version $F_H$ and $\N_H$ by using a nonhomogeneous Littlewood-Paley decomposition as follows
$$
\|f\|_{F^s}^2 = \sum_{H\in\D} H^{2s} \|P_H f\|_{F_H}^2 \,
$$
$$
\|f\|_{\N^s}^2 = \sum_{H\in\D} H^{2s} \|P_H f\|_{\N_H}^2.
$$
We define next a time localized version of those spaces. For $T>0$ and $Y=F^s$ or $Y=\N^s$, the space $Y(T)$ is defined by its norm
$$
\|f\|_{Y(T)} = \inf \{ \|\widetilde{f}\|_{Y} : \widetilde{f} : \R^3\to\R\text{ and } \widetilde{f}|_{[-T,T]\times\R^2} = f\}.
$$
For $s\ge 0$ and $T>0$ we define the Banach spaces for the initial data $E^s$ by
$$
\|\phi\|_{E^s} = \|\cro{h(\xi,\mu)}^s\cdot \hat{\phi}\|_{L^2_{\xi,\mu}},
$$
and their intersections are denoted by $E^\infty=\bigcap_{s\ge 0}E^s$. Finally, the associated energy spaces $B^s(T)$ are endowed with norm
$$
\|f\|_{B^s(T)}^2 = \|P_1f(0,\cdot)\|_{L^2_{xy}}^2 + \sum_{H\in\D\setminus\{1\}} H^{2s} \sup_{t_H\in[-T,T]} \|P_Hf(t_H,\cdot)\|_{L^2_{xy}}^2.
$$

\subsection{Properties of the function spaces}
In this section, we state without proof some important results related to the short time function spaces introduced in the previous section. They all have been proved in different contexts in \cite{IKT}-\cite{KP}-\cite{G}.

The $F^s(T)$ and $\N^s(T)$ spaces enjoy the following linear properties.
\begin{lemma}\label{LinftyEF}
  Let $T>0$ and $s\ge 0$. Then it holds that
  \begin{equation}
    \|f\|_{L^\infty_TE^s}\lesssim \|f\|_{F^s(T)}
  \end{equation}
  for all $f\in F^s(T)$.
\end{lemma}

\begin{proposition}
  Assume $T\in (0,1]$ and $s\ge 0$. Then we have that
 \begin{equation}\label{FBN}
   \|u\|_{F^s(T)} \lesssim \|u\|_{B^s(T)} + \|f\|_{\N^s(T)}
 \end{equation}
 for all $u\in B^s(T)$ and $f\in \N^s(T)$ satisfying
 $$
 \partial_tu + D^\alpha_x\partial_xu+\partial_{xyy}u = f \textrm{ on } [-T,T]\times\R^2.
 $$
\end{proposition}

We will also need the following technical results.
\begin{lemma}\label{fond}
  Let $H,H_1\in\D$ be given. Then it holds that
  $$
  H^{\beta/2} \left\| \varphi_{\le \lfloor H^\beta\rfloor}(\tau-\omega(\zeta) \int_\R |\phi(\tau',\zeta) H^{-\beta}(1+H^{-\beta}|\tau-\tau'|)^{-4}d\tau'\right\|_{L^2_{\tau\zeta}} \lesssim \|\phi\|_{X_{H_1}}
  $$
  and
  $$
  \sum_{L> \lfloor H^\beta\rfloor} L^{1/2} \left\| \varphi_{L}(\tau-\omega(\zeta) \int_\R |\phi(\tau',\zeta) H^{-\beta}(1+H^{-\beta}|\tau-\tau'|)^{-4}d\tau'\right\|_{L^2_{\tau\zeta}} \lesssim \|\phi\|_{X_{H_1}}
  $$
  for all $\phi\in F_{H_1}$.
\end{lemma}

\begin{corollary}\label{cor-fond}
  Let $\widetilde{t}\in\R$ and $H,H_1\in\D$ be such that $H\gg H_1$. Then it holds that
  $$
  H^{\beta/2} \|\varphi_{\le \lfloor H^\beta\rfloor}\F(\varphi_1(H^\beta(\cdot-\widetilde{t}))f)\|_{L^2_{\tau\zeta}} \lesssim \|f\|_{F_{H_1}}
  $$
  and
  $$
  \sum_{L>\lfloor H^\beta\rfloor} L^{1/2} \|\varphi_{L}\F(\varphi_1(H^\beta(\cdot-\widetilde{t}))f)\|_{L^2_{\tau\zeta}} \lesssim \|f\|_{F_{H_1}}
  $$
  for all $f\in F_{H_1}$.
\end{corollary}

\begin{lemma}\label{techX}
  Let $H\in \D$ and $I\subset \R$ an interval. Then
  $$
  \sup_{L\in\D} L^{1/2} \|\varphi_L(\tau-\omega(\zeta)) \F(1_I(t)f)\|_{L^2(\R^3)} \lesssim \|\F(f)\|_{X_H},
  $$
  for all $f$ such that $\F(f)\in X_H$.
\end{lemma}

\section{Strichartz estimates}\label{stri-sec}

For $1\le \alpha \le 2$ we set $B=\alpha (\alpha +1)/2$, and for $\delta >0$ small enough, let us define
$$
A_\delta = \{(\xi,\mu)\in \R^2: (B-\delta)|\xi|^\alpha \le \mu^2\le (B+\delta)|\xi|^\alpha \}.
$$
We also consider a function $\rho\in C^\infty(\R,[0,1])$ satisfying $\rho=0$ on $[-\frac 12,\frac 12]$ and $\rho(\xi)=1$ for $|\xi|\ge 1$. We set $\rho_\delta(\xi)=\rho(\xi/\delta)$ so that $\rho_\delta(B-\frac{\mu^2}{|\xi|^\alpha })$ is a smooth version of $1_{A_\delta^c}$ in the sense that
\begin{equation}\label{rho}
\forall (\xi,\mu)\in \R^\ast\times\R,\quad \rho_\delta \left(B-\frac{\mu^2}{|\xi|^\alpha }\right)1_{A_\delta^c}(\xi,\mu) = 1_{A_\delta^c}(\xi,\mu),
\end{equation}
and $\rho_\delta \left(B-\frac{\mu^2}{|\xi|^\alpha }\right) = 0$ on $A_{\delta/2}$.
The main result of this section is the following.
\begin{proposition}\label{stri}
Let $N\in\D$ and $\delta\in (0,1)$. Assume that $(p,q)$ satisfies $\frac 1q=\theta\frac {1-\eps}2$ and $\frac 1p=\frac{1-\theta}2$ for some $\theta\in [0,1)$ and  $\eps>0$  small enough. Then it holds that
\begin{equation}\label{main-stri}
\|P_N^xP_{A_\delta^c}U(t)\phi\|_{L^q_tL^p_{xy}}\lesssim N^{\theta(\eps(\alpha +1) -\frac \alpha 4)}\|\phi\|_{L^2}.
\end{equation}
for any $\phi\in L^2(\R^2)$.
\end{proposition}

\begin{remark}
  We notice that in the case $\alpha=2$ and $\theta=1/2+$, estimate \eqref{main-stri} was already used in \cite{MP} and is a direct consequence of a more general theorem related to homogeneous polynomial hypersurfaces proved by Carbery, Kenig and Ziesler \cite{CKZ}. However, this result does not apply as soon as $\alpha<2$ since the symbol $\omega$ defined in \eqref{pomega} is no more homogeneous.
\end{remark}

To prove Proposition \ref{stri}, we will need the following result.
\begin{lemma}\label{It}
  Let $N\in\D\setminus\{1\}$. Then
  $$
  I_t(x,y) = \int_{\R^2} e^{i(t\omega(\xi,\mu)+x\xi+y\mu)} \varphi_N(\xi)\rho_\delta\left(B-\frac{\mu^2}{|\xi|^\alpha }\right) d\xi d\mu
  $$
  satisfies
  \begin{equation}\label{It.0}
  \|I_t\|_{L^\infty_{xy}} \lesssim N^{-\alpha /2}|t|^{-1},
  \end{equation}
  for all $t\in\R^\ast$ and $\delta\in (0,1)$.
\end{lemma}

\begin{proof}
  First, recall that the semi-convergent integral $I_t$ may be understood as
  \begin{align}
  I_t(x,t) &= \lim_{M\to\infty} \int_{\R^2} e^{i(t\omega(\xi,\mu)+x\xi+y\mu)} \varphi_N(\xi)\varphi_{\le M}(\mu)\rho_\delta\left(B-\frac{\mu^2}{|\xi|^\alpha }\right) d\xi d\mu \notag\\
  &=\lim_{M\to\infty} (I_t^++I_t^-) \label{It.1}
  \end{align}
  with
  \begin{equation}\label{It.2}
  I_t^{\pm}(x,y) = \int_{\R^2} e^{i(t\omega(\xi,\mu)+x\xi+y\mu)} \varphi_N(\xi)\varphi_{\le M}(\mu)\rho_\delta^{\pm}\left(B-\frac{\mu^2}{|\xi|^\alpha }\right) d\xi d\mu,
  \end{equation}
  and $\rho_\delta^{\pm} = \rho_\delta 1_{\R_\pm}$. We are going to bound $|I_t^\pm|$, uniformly in $x,y$ and $M$. Let $\eps\in (0,1)$ be a small number to be chosen later and define
  \begin{align*}
    B_\eps^+ &= \{\xi\in\R : (\alpha +1+B-\frac \delta 2)|t||\xi|^\alpha < (1-\eps) |x|\},\\
    B_\eps^- &= \{\xi\in\R : (\alpha +1+B+\frac \delta 2)|t||\xi|^\alpha  > (1+\eps) |x|\}.
  \end{align*}
  Then $I_t^\pm$ may be decomposed as
  \begin{align}
    I_t^\pm(x,y) &= \left(\int_{B_\eps^\pm\times \R} + \int_{(B_\eps^\pm)^c\times\R}\right) e^{i(t\omega(\xi,\mu)+x\xi+y\mu)} \varphi_N(\xi)\varphi_{\le M}(\mu)\rho_\delta^{\pm}\left(B-\frac{\mu^2}{|\xi|^\alpha }\right) d\xi d\mu \notag\\
    &:= I_{t,1}^\pm(x,y) + I_{t,2}^\pm(x,y). \label{It.3}
  \end{align}
  We estimate $I_{t,1}^\pm$ and rewrite it as
  \begin{equation}\label{It.4}
  I_{t,1}^\pm (x,y) = \int_\R e^{iy\mu}\varphi_{\le M}(\mu) \left(\int_{B_\eps^\pm} e^{i\psi_1(\xi)} \lambda_1^\pm(\xi) d\xi\right) d\mu ,
  \end{equation}
  where the phase function $\psi_1$ is defined by $\psi_1(\xi) = x\xi + t\xi(|\xi|^\alpha +\mu^2)$, and where $\lambda_1^\pm(\xi) = \varphi_N(\xi)\rho_\delta^\pm \left(B-\frac{\mu^2}{|\xi|^\alpha }\right)$. Then we easily check that
  \begin{equation}\label{It.5}
    |\psi_1'(\xi)|\gtrsim |t|(N^\alpha +\mu^2) \text{ on } B_\eps^\pm\cap \supp(\lambda_1^\pm).
  \end{equation}
  Indeed if we assume that $\xi\in B_\eps^+\cap \supp(\lambda_1^+)$, then it holds $\mu^2<(B-\frac\delta 2)|\xi|^\alpha $ and $|x|> \frac{\alpha +1 +B-\delta/2}{1-\eps}|t||\xi^\alpha |$, from which it follows $|x|> (1-\eps)^{-1}|t|((\alpha +1)|\xi|^\alpha +\mu^2)$. Since $\psi_1'(\xi) = x+t((\alpha+1)|\xi|^\alpha +\mu^2)$, we deduce $|\psi_1'(\xi)|\gtrsim \max(|x|,|t|((\alpha +1)|\xi|^\alpha +\mu^2))$. A similar argument leads also to (\ref{It.5}) for $\xi\in B_\eps^-\cap\supp(\lambda_1^-)$. Moreover, observe that
  \begin{equation}\label{It.6}
   \| \psi_1''\|_{L^{\infty}}\lesssim |t|N^{\alpha -1} , \quad \| \lambda_1^\pm\|_{L^\infty}\lesssim 1,\quad \|\lambda_1^\pm\|_{L^1}\lesssim N,
  \end{equation}
  and
  \begin{equation}\label{It.7}
    \|(\lambda_1^\pm)'\|_{L^1}\lesssim N^{-1}\left(\int_\R|\varphi'(\xi/N)|d\xi + \int_\R |\varphi_N(\xi)| \frac{\mu^2}{|\xi|^{\alpha +1}}\left|(\rho_\delta^\pm)'\left(B-\frac{\mu^2}{|\xi|^\alpha }\right)\right|d\xi\right) \lesssim 1.
  \end{equation}

   Using (\ref{It.5})-(\ref{It.6})-(\ref{It.7}), an integration by parts yields
  \begin{align*}
  \left|\int_{B_\eps^\pm} e^{i\psi_1}\lambda_1^\pm\right| &= \left|\int_{B_\eps^\pm}(e^{i\psi_1})'\frac{\lambda_1^\pm}{\psi_1'}\right| \, d\xi \\
  &\lesssim \frac{\|\lambda_1^\pm\|_{L^\infty}}{|t|(N^\alpha +\mu^2)} + \int_\R\left( \frac{|(\lambda_1^\pm)'|}{|t|(N^\alpha +\mu^2)} + \frac{|\lambda_1^\pm \psi_1''|}     {(t(N^\alpha +\mu^2))^2}\right)\, d\xi \\
  &\lesssim \frac{\|\lambda_1^\pm \|_{L^{\infty}}}{|t|(N^\alpha +\mu^2)} +  \frac{\| (\lambda_1^\pm)'\|_{L^1}}{|t|(N^\alpha +\mu^2)}+ \frac{ \| \lambda_1^\pm \|_{L^1}
  \| \psi_1''\|_{L^{\infty}}}{(t(N^\alpha +\mu^2))^2} \\
  &\lesssim |t|^{-1}(N^\alpha +\mu^2)^{-1} + |t|N N^{\alpha -1} (t(N^{\alpha }+\mu^2))^{-2} \\
  &\lesssim |t|^{-1}(N^\alpha +\mu^2)^{-1}.
  \end{align*}
  Coming back to (\ref{It.4}) we infer
  \begin{equation}\label{It.8}
    |I_{t,1}^{\pm}(x,y)| \lesssim |t|^{-1}\int_\R \frac{d\mu}{N^\alpha +\mu^2}\lesssim N^{-\alpha /2}|t|^{-1}.
  \end{equation}
  It remains to estimate $I_{t,2}^\pm$. Using that
  \begin{equation}\label{It.09}
    \int_\R e^{i(t\xi\mu^2+y\mu)}d\mu = \frac{\sqrt{\pi}}{|t\xi|^{1/2}} e^{-i\frac{y^2}{4t\xi} + i\frac{\pi}4 \sgn(\xi)},
  \end{equation}
  we get
  \begin{multline*}
  I_{t,2}^+(x,y) = \int_{(B_\eps^+)^c\times\R^2} \frac{\sqrt{\pi}}{|t\xi|^{1/2}} e^{i(x\xi+t\xi|\xi|^\alpha -\frac{v^2}{4t\xi}+\frac\pi 4\sgn(\xi))} \\
  \times \F^{-1}_\mu(\rho_\delta^+ (B-\frac{\mu^2}{|\xi|^\alpha }))(u) \F^{-1}(\varphi_{\le M})(y-v-u) \varphi_N(\xi) d\xi dudv.
  \end{multline*}
  Performing the change of variables $u\to |\xi|^{-\alpha/2}u$, a dilatation argument leads to
  \begin{multline}\label{It.9}
  I_{t,2}^+(x,y) = \int_{\R^2} \F^{-1}(\rho_\delta^+(B-\mu^2))(v) \F^{-1}(\varphi_{\le M})(y-u)\\
  \times \left(\int_{(B_\eps^+)^c} e^{i(x\xi+t\xi|\xi|^\alpha -\frac{(u-|\xi|^{-\frac \alpha 2}v)^2}{4t\xi}+\frac\pi 4\sgn(\xi))} \frac{\sqrt{\pi}}{|t\xi|^{1/2}} \varphi_N(\xi)d\xi\right) dudv.
  \end{multline}
  Since $\rho_\delta^+(B-\mu^2)\in \mathcal{D}(\R)$, we infer
  \begin{equation}\label{It.10}
  |I_{t,2}^+(x,y)| \lesssim \sup_{u,v\in\R} \cro{v}^{-2}\left|\int_{(B_\eps^+)^c} e^{i\psi_2}\lambda_2\right| := \sup_{u,v\in\R} \cro{v}^{-2} |J^+(u,v)|
  \end{equation}
  where the new phase function $\psi_2$ is defined by $\psi_2(\xi) = x\xi+t\xi|\xi|^\alpha -\frac{u^2}{4t\xi}+\frac \pi 4\sgn(\xi)$, and
  $$
  \lambda_2(\xi) = \frac{\sqrt{\pi}}{|t\xi|^{1/2}} e^{i\left(\frac{uv}{2t\xi|\xi|^{\alpha /2}} - \frac{v^2}{4t\xi|\xi|^\alpha}\right)} \varphi_N(\xi).
  $$
  We argue similarly to estimate $I_{t,2}^-$, except that we rewrite  $\rho_\delta^-(B-\mu^2)$ as $\rho_\delta^-(B-\mu^2)  = (\rho_\delta^-(B-\mu^2)-1)+1$. Hence we have,
  \begin{multline}\label{It.11}
    I_{t,2}^-(x,y) = \int_{\R^2} \F^{-1}_\mu(\rho_\delta^-(B-\mu^2)-1)(v)\F^{-1}(\varphi_{\le M})(y-u) \left(\int_{(B_\eps^-)^c}e^{i\psi_2}\lambda_2\right)dudv\\
    + \int_\R \F^{-1}(\varphi_{\le M})(y-u)\left(\int_{(B_\eps^-)^c} e^{i\psi_2}\lambda_3\right)
  \end{multline}
  with $\lambda_3(\xi) = \frac{\sqrt{\pi}}{|t\xi|^{1/2}} \varphi_N(\xi)$. Since $\rho_\delta^-(B-\mu^2)-1\in\mathcal{D}(\R)$, estimate (\ref{It.10}) together with (\ref{It.11}) lead to
  \begin{equation}\label{It.12}
  |I_{t,2}^+(x,y)|+|I_{t,2}^-(x,y)| \lesssim \sup_{u,v\in\R}\left(\cro{v}^{-2} (|J^+(u,v)|+|J^-(u,v)|)+|K(u)|\right),
  \end{equation}
  where $J^-(u,v)=\int_{(B_\eps^-)^c} e^{i\psi_2}\lambda_2$ and $K(u) = \int_{(B_\eps^-)^c} e^{i\psi_2}\lambda_3$. Noticing that
  \begin{equation}\label{It.13}
  \|\lambda_2\|_{L^1} + \|\lambda_3\|_{L^1} \lesssim N^{1/2} |t|^{-1/2},
  \end{equation}
  we get
  \begin{equation}
    |I_{t,2}^\pm(x,y)|\lesssim N^{1/2} |t|^{-1/2},
  \end{equation}
  which is acceptable as soon as $|t|< N^{-(\alpha +1)}$. Therefore we assume now that $|t|\ge N^{-(\alpha +1)}$. Observe that since (\ref{It.9}) and (\ref{It.11}) also holds for $I^\pm_{t,1}$ with $(B_\eps^\pm)^c$ replaced with $B_\eps^\pm$, we deduce from (\ref{It.13}) that
  \begin{equation}\label{It.14}
   |I_t(x,y)|\lesssim N^{1/2}|t|^{-1/2},
  \end{equation}
  for any $(t,x,y)\in\R^\ast\times\R^2$.
  Differentiating the phase function we get
  \begin{align}
    \psi_2'(\xi) &= x+(\alpha +1)t|\xi|^\alpha + \frac{u^2}{4t\xi^2},\label{It.15}\\
    \psi_2''(\xi) &= \alpha (\alpha +1)t\sgn(\xi)|\xi |^{\alpha -1} - \frac{u^2}{2t\xi^3}.\label{It.015}
  \end{align}
   Let $\gamma\in (0,1)$ be a small parameter that we will choose later, and define
  $$
  C_\gamma = \{\xi\in\R : (1-\gamma)\alpha (\alpha +1) |t||\xi |^{\alpha -1} < \frac{u^2}{2|t\xi^3|} < (1+\gamma)\alpha (\alpha +1)|t||\xi |^{\alpha -1}.
  $$
  We decompose $J^\pm$ as
  \begin{equation}\label{It.16}
  J^\pm(u,v) = \left(\int_{(B_\eps^\pm)^c\cap C_\gamma} + \int_{(B_\eps^\pm)^c\cap C_\gamma^c}\right)e^{i\psi_2}\lambda_2 := J^\pm_1(u,v)+J^\pm_2(u,v).
  \end{equation}
  From the definition of $C_\gamma$, we have $|\psi_2''(\xi)|\gtrsim |t|N^{\alpha -1} \vee \frac{u^2}{|t|N^3}$ for $\xi\in C_\gamma^c$. Moreover, we have $\|\lambda_2\|_{L^\infty}\lesssim |t|^{-1/2} N^{-1/2}$ and straightforward calculations lead to
  \begin{equation}\label{It.17}
  \|\lambda_2'\|_{L^1} \lesssim |t|^{-1/2} N^{-1/2}+ |t|^{-3/2} N^{-\alpha +3/2}\cro{v}^2 + |t|^{-3/2}N^{-\alpha /2 +5/2}|uv|.
  \end{equation}
  The Van der Corput lemma applies and provides
  \begin{equation}\label{It.18}
  |J_2^\pm(u,v)| \lesssim \left(|t|N^{(\alpha -1)}\vee \frac{u^2}{|t|N^3}\right)^{-1/2} (\|\lambda_2\|_{L^\infty} + \|\lambda_2'\|_{L^1}) \lesssim N^{-\alpha/2}|t|^{-1}\cro{v}^2.
  \end{equation}
  To estimate $J_1^\pm$, we will take advantage of the first derivative of $\psi_2$ given by (\ref{It.15}). Let $\xi\in C_\gamma$. Then we easily see that
  $$
  ((\alpha +1) +B -\gamma B)|t||\xi|^\alpha < \left| (\alpha +1)|t|\xi|^\alpha + \frac{u^2}{4t\xi^2}\right| < ((\alpha +1) +B +\gamma B)|t||\xi|^\alpha.
  $$
  If $\xi\in (B_\eps^+)^c$, then
  $$
  |x| < \frac{\alpha +1 +B -\delta/2}{1-\eps}|t||\xi|^\alpha < (1-\eps)^{-1} \frac{\alpha+1+B-\delta/2}{\alpha +1 +B -\gamma B}\, \left|(\alpha +1)|t|\xi|^\alpha +\frac{u^2}{4t\xi^2}\right|
  $$
  and if $\xi\in (B_\eps^-)^c$, we have
  $$
  |x|> \frac{\alpha +1 +B +\delta/2}{1+\eps} |t||\xi^\alpha| > (1+\eps)^{-1}\frac{\alpha +1 +B +\delta/2}{\alpha +1+B +\gamma B} \, \left|(\alpha +1)|t|\xi|^\alpha +\frac{u^2}{4t\xi^2}\right|.
  $$
  Since we can always choose $\eps,\gamma>0$ small enough so that $(1-\eps)^{-1} \frac{\alpha +1+B -\delta/2}{\alpha +1+B -\gamma B}<1$ and $(1+\eps)^{-1}\frac{\alpha +1+B+\delta/2}{\alpha +1+B +\gamma B}>1$, we infer
  \begin{equation}\label{It.19}
  |\psi_2'(\xi)|\gtrsim |x|\vee |t|N^\alpha \vee \frac{u^2}{|t|N^2} \text{ on } (B_\eps^\pm)^c\cap C_\gamma.
  \end{equation}
  Therefore $J_1^\pm$ is estimated thanks to (\ref{It.17})-(\ref{It.19}) and integration by parts as follows
  \begin{align}
    |J_1^\pm(u,v)| &= \left|\int_{(B_\eps^\pm)^c\cap C_\gamma} (e^{i\psi_2})'\frac{\lambda_2}{\psi_2'}\right| \notag\\
    &\lesssim \left(|t|N^\alpha \vee \frac{u^2}{|t|N^2}\right)^{-1} (\|\lambda_2\|_{L^\infty}+\|\lambda_2'\|_{L^1}) + \left(|t|N^\alpha \vee \frac{u^2}{|t|N^2}\right)^{-2} \|\psi_2''\|_{L^\infty(C_\gamma)}\|\lambda_2\|_{L^1} \notag \\
    &\lesssim |t|^{-3/2}N^{-\alpha/2} \cro{v}^2 \lesssim N^{-\alpha/2}|t|^{-1}\cro{v}^2. \label{It.20}
  \end{align}
  Combining (\ref{It.16})-(\ref{It.18})(\ref{It.20}) we deduce
  \begin{equation}\label{It.21}
  \sup_{u,v\in\R}(\cro{u}^{-2} (|J^+(u,v)|+|J^-(u,v)|)\lesssim N^{-1/2}|t|^{-1},
  \end{equation}
  as desired. Estimates for $K$ are similar, since (\ref{It.17}) is replaced with
  $$
  \|\lambda_3'\|_{L^1} \lesssim |t|^{-1/2}N^{-1/2}.
  $$
  We obtain the bound
  \begin{equation}\label{It.22}
    \sup_{u,v\in\R} |K(u)|\lesssim N^{-1/2}|t|^{-1}.
  \end{equation}
  Combining (\ref{It.1})-(\ref{It.2})-(\ref{It.3})-(\ref{It.8})-(\ref{It.12})-(\ref{It.21})-(\ref{It.22}) we complete the proof of Lemma \ref{It}.
\end{proof}

\begin{proof}[Proof of Proposition \ref{stri}] The case $N=1$ is straightforward, therefore we assume $N\ge 2$. Interpolating estimates (\ref{It.0}) and (\ref{It.14}) we get for any $\eps\in (0,1)$
\begin{equation}\label{stri.1}
\|I_t\|_{L^\infty_{xy}} \lesssim N^{-\frac \alpha 2+\frac{\eps (\alpha +1)}{2}} |t|^{-1+\frac\eps 2}.
\end{equation}
On the other hand, we get from (\ref{rho}) that
$$
P_N^xP_{A_\delta^c}U(t)\phi = I_t \ast_{xy} (P_{A_\delta^c}\phi).
$$
Thus, thanks to Young inequality and estimate (\ref{stri.1}), we infer
$$
\|P_N^xP_{A_\delta^c}U(t)\phi\|_{L^\infty_{xy}} \lesssim N^{-\frac \alpha 2+\frac{\eps (\alpha +1)}{2}} |t|^{-1+\frac \eps 2} \|P_{A_\delta^c}\phi\|_{L^1},
$$
for any $t\in\R^\ast$. Therefore, by interpolation with the straightforward equality $\|U(t)\phi\|_{L^2_{xy}}=\|\phi\|_{L^2}$ we deduce that for any $\theta\in [0,1)$,
$$
\|P_N^xP_{A_\delta^c}U(t)\phi\|_{L^p_{xy}} \lesssim N^{\theta(\frac{\eps (\alpha +1)}{2}-\frac \alpha 2)} |t|^{\theta(\frac \eps 2-1)} \|\phi\|_{L^{p'}},
$$
where $\frac 1p+\frac 1{p'}=1$ and $\frac 1p=\frac{1-\theta}2$. Remark that we exclude the case $\theta=1$ because the operator $P_{A_\delta^c}$ is not continuous on $L^1(\R^2)$. The previous estimate combined with the triangle inequality and Hardy-Littlewood-Sobolev theorem lead to
\begin{equation}\label{stri.2}
  \left\|P_N^xP_{A_\delta^c} \int_\R U(t-t')f(t')dt'\right\|_{L^q_tL^p_{xy}} \lesssim N^{\theta(\eps (\alpha+1)-\frac \alpha 2)} \|f\|_{L^{q'}_tL^{p'}_{xy}},
\end{equation}
for all $f\in \mathcal{S}(\R^3)$, where $\frac 1q+\frac 1{q'}=1$ and $\frac 2q = 1-\frac\eps 2$. Estimate (\ref{stri}) is then obtained from (\ref{stri.2}) by the classical Stein-Thomas argument.
\end{proof}

\begin{corollary}\label{l4}
  Assume $\delta\in (0,1)$, $H,N\in\D$ and $f\in X_H$. Then for all $s>-\alpha/8$, it holds that
  \begin{equation}\label{l4Loc}
  \|P_{A_\delta^c} P_N^x \F^{-1}(f)\|_{L^4_{txy}} \lesssim N^s \|f\|_{X_H}.
  \end{equation}
\end{corollary}

\begin{proof}
  We apply Proposition \ref{stri} with $\theta=\frac{1}{2-\eps}$ and obtain
  $$
  \|P_N^xP_{A_\delta^c}U(t)\phi\|_{L^{4+}_{txy}} \lesssim N^{(-\alpha/8)+} \|\phi\|_{L^2}
  $$
  for any $\phi\in L^2(\R^2)$. Setting $f^\sharp(\theta,\xi,\mu) = f(\theta+\omega(\xi,\mu),\xi,\mu)$ it follows then from Minkowski and Cauchy-Schwarz in $\theta$ that
\begin{align*}
  \|P_{A_\delta^c}P_N^x\F^{-1}(f)\|_{L^{4+}}&\lesssim \|U(t)P_{A_\delta^c}P_N^x \F^{-1}(f^\sharp)\|_{L^{4+}}\\
  &\lesssim N^{(-\alpha/8)+} \int_\R \|f^\sharp(\theta,\zeta)\|_{L^2_\zeta}d\theta\\
  &\lesssim N^{(-\alpha/8)+} \sum_{L\in\D} L^{1/2} \|\varphi_L(\theta)f^\sharp(\theta,\zeta)\|_{L^2_{\theta\zeta}}\\
  &\lesssim N^{(-\alpha/8)+} \|f\|_{X_H}.
\end{align*}
Interpolating this with the trivial bound $\|\F^{-1}(f)\|_{L^2_{txy}}\lesssim \|f\|_{X_H}$ we conclude the proof of Corollary \ref{l4}.
\end{proof}

We conclude this section by stating a global Strichartz estimate that will not be used in the proof of Theorem \ref{theo-main}, but that may be of independent interest for future considerations.

\begin{proposition}
Let $N\in\D$. Assume that $(p,q)$ satisfies $\frac 1q=\frac {5\theta}{12}$ and $\frac 1p=\frac{1-\theta}2$ for some $\theta\in [0,1]$. Then it holds that
\begin{equation}\label{striGlo}
\|P_N^xU(t)\phi\|_{L^q_tL^p_{xy}}\lesssim N^{-\frac \theta 6(\alpha -\frac 12)}\|\phi\|_{L^2}.
\end{equation}
for any $\phi\in L^2(\R^2)$.
\end{proposition}

\begin{proof}
  As in the proof of Proposition \ref{stri}, it suffices to show that
  $$
  \widetilde{I}_t(x,y) := \int_{\R^2} e^{i(t\omega(\xi,\mu)+x\xi+y\mu)} \varphi_N(\xi) \varphi_{\le M}(\mu) d\xi d\mu
  $$
  satisfies
  $$
    |\widetilde{I}_t(x,y)|\lesssim N^{\frac 16-\frac \alpha 3} |t|^{-5/6},
  $$
  with an implicit constant that does not depend on $M\in\D$. Thanks to \eqref{It.09} we may rewrite $\widetilde{I}_t$ as
  $$
  \widetilde{I}_t(x,y) = \int_\R \F^{-1}(\varphi_{\le M})(y-u) \left(\int_\R \frac{\sqrt{\pi}}{|t\xi|^{1/2}} e^{i\psi_2(\xi)} \varphi_N(\xi) d\xi\right) du
  $$
  where $\psi_2(\xi)=x\xi+t\xi|\xi|^\alpha-\frac{u^2}{4t\xi}+\frac\pi 4\sgn(\xi)$ was defined in \eqref{It.10}. Since the third derivative of $\psi_2$ is given by $\psi_2'''(\xi) = \alpha(\alpha-1)(\alpha+1) t|\xi|^{\alpha-2} + \frac{3u^2}{2t\xi^4}$, the Van der Corput lemma implies in the case $\alpha>1$ that
  $$
  |\widetilde{I}_t(x,y)|\lesssim (|t|N^{\alpha-2})^{1/3} (|t|N)^{-1/2}\sim N^{\frac 16-\frac\alpha 3}|t|^{-5/6}
  $$
  as desired. Now consider the case $\alpha=1$. In the region where $|t|\sim \frac{u^2}{|t|N^3}$, we get directly that $|\psi_2'''|\sim |t|N^{-1}$ as previously. Therefore we may assume $|t|\not\sim \frac{u^2}{|t|N^3}$. From \eqref{It.015} we deduce $|\psi_2''|\gtrsim |t|$ which combined with the Van der Corput lemma provides
  \begin{equation}\label{striGlo.1}
  |\widetilde{I}_t(x,y)| \lesssim |t|^{-1/2} (|t|N)^{-1/2}\sim N^{-1/2} |t|^{-1}.
  \end{equation}
  On the other hand, we have the trivial bound
  \begin{equation}\label{striGlo.2}
    |\widetilde{I}_t(x,y)| \lesssim \int_\R \frac{\varphi_N(\xi)}{|t\xi|^{1/2}}d\xi \lesssim N^{1/2}|t|^{-1/2}.
  \end{equation}
  Gathering \eqref{striGlo.1}-\eqref{striGlo.2} we infer
  $$
  |\widetilde{I}_t(x,y)|\lesssim (N^{-1/2}|t|^{-1})^{2/3} (N^{1/2}|t|^{-1/2})^{1/3} \sim N^{-1/6}|t|^{-5/6},
  $$
  which concludes the proof of \eqref{striGlo}.
\end{proof}

\begin{remark}
  It follows by applying estimate \eqref{striGlo} with $\theta=1/2$ that
  $$
  \|P_N^xU(t)\phi\|_{L^{24/5}_tL^4_{xy}}\lesssim N^{-\frac 1{12}(\alpha -\frac 12)}\|\phi\|_{L^2}.
  $$
  Therefore, arguing as in the proof of Corollary \ref{l4} we infer that for all $f\in X_H$ such that $\supp \F^{-1}(f)\subset [0,T]\times\R^2$ for some $T\in (0,1]$, we have
  \begin{equation}\label{l4Glo}
    \|P_N^x\F^{-1}(f)\|_{L^4_{txy}} \lesssim N^{-\frac 1{12}(\alpha -\frac 12)}\|f\|_{X_H}.
  \end{equation}
  Consequently, \eqref{l4Loc} can be viewed as an improvement of estimate \eqref{l4Glo} since outside the curves $\mu^2=B|\xi|^\alpha$, it allows to recover $\frac \alpha 8$ derivatives instead of $\frac 1{12}(\alpha-\frac 12)$ derivatives in $L^4$.
\end{remark}

\section{$L^2$ bilinear estimates}\label{sec-l2bil}
For $H,N,L\in\D$, let us define $D_{H,N,L}$ and $D_{H,\infty,L}$ by
\begin{equation}\label{DNML}
D_{H,N,L} = \{(\tau,\xi,\mu)\in\R^3: \xi\in I_N, (\xi,\mu)\in \Delta_H \textrm{ and } |\tau+\omega(\xi,\mu)|\le L\},
\end{equation}
and
\begin{equation}\label{DNinfL}
D_{H,\infty,L} = \{(\tau,\xi,\mu)\in\R^3 : (\xi,\mu)\in \Delta_H \textrm{ and } |\tau+\omega(\xi,\mu)|\le L\} = \bigcup_{N\in\D} D_{H,N,L}.
\end{equation}

\begin{proposition}\label{prop.bilinear}
 Assume that $H_i, N_i, L_i\in \D$ are dyadic numbers and $f_i:\R^3\to\R_+$ are $L^2$ functions for $i=1,2,3$.
 \begin{enumerate}
 \item If $f_i$ are supported in $D_{H_i,\infty,L_i}$ for $i=1,2,3$, then
 \begin{equation}\label{prop.bilinear.00}
  \int_{\R^3} (f_1\ast f_2)\cdot f_3 \lesssim H_{min}^{\frac 1{2\alpha}+\frac 14} L_{min}^{1/2} \|f_1\|_{L^2}\|f_2\|_{L^2}\|f_3\|_{L^2}.
 \end{equation}
 \item Let us suppose that $H_{min}\ll H_{max}$ and $f_i$ are supported in $D_{H_i,\infty,L_i}$ for $i=1,2,3$. If $(H_i,L_i)=(H_{min},L_{max})$ for some $i\in\{1,2,3\}$ then
 \begin{equation}\label{prop.bilinear.0}
  \int_{\R^3} (f_1\ast f_2)\cdot f_3 \lesssim H_{max}^{-1/2} H_{min}^{1/4} L_{min}^{1/2}L_{max}^{1/2} \|f_1\|_{L^2}\|f_2\|_{L^2}\|f_3\|_{L^2}.
 \end{equation}
 Otherwise we have
 \begin{equation}\label{prop.bilinear.001}
  \int_{\R^3} (f_1\ast f_2)\cdot f_3 \lesssim H_{max}^{-1/2} H_{min}^{1/4} L_{min}^{1/2}L_{med}^{1/2} \|f_1\|_{L^2}\|f_2\|_{L^2}\|f_3\|_{L^2}.
 \end{equation}
 \item If $H_{min}\sim H_{max}$ and $f_i$ are supported in $D_{H_i,N_i,L_i}$ for $i=1,2,3$, then
 \begin{equation}\label{prop.bilinear.01}
  \int_{\R^3} (f_1\ast f_2)\cdot f_3 \lesssim N_{max}^{-\alpha/2} H_{min}^{(1/4)+} L_{med}^{1/2}L_{max}^{1/2} \|f_1\|_{L^2}\|f_2\|_{L^2}\|f_3\|_{L^2}.
 \end{equation}
 \end{enumerate}
\end{proposition}

Before proving Proposition \ref{prop.bilinear} we give a technical lemma.
\begin{lemma}\label{tech}
  Assume $0<\delta<1$. Then we have that
  \begin{equation}\label{tech.1}
    h(\zeta_1+\zeta_2) \le |h(\zeta_1)-h(\zeta_2)| + f(\delta) \max(h(\zeta_1), h(\zeta_2)),
  \end{equation}
  for all $\zeta_i=(\xi_1,\mu_i)\in \R^2$, $i=1,2$ satisfying
  \begin{equation}\label{tech.2}
    (\xi_1,\mu_1), (\xi_2,\mu_2) \in A_\delta
  \end{equation}
  and
  $$
    \xi_1\xi_2<0 \textrm{ and } \mu_1\mu_2<0,
  $$
  and where $f$ is a continuous function on $[0,1]$ satisfying $\lim_{\delta\to 0}f(\delta)=0$.
\end{lemma}

\begin{proof}
  Without loss of generality, we may assume
  \begin{equation}\label{tech.3}
  \xi_1>0,\ \mu_1>0,\ \xi_2<0,\ \mu_2<0\textrm{ and } h(\zeta_1)\ge h(\zeta_2).
  \end{equation}
  Thus, it suffices to prove that
  \begin{equation}\label{tech.4}
    (\alpha+1)(|\xi_1+\xi_2|^\alpha+|\xi_2|^\alpha-|\xi_1|^\alpha) + 2\mu_2(\mu_1+\mu_2)\le f(\delta)h(\zeta_1).
  \end{equation}
  Thanks to \eqref{tech.2} and \eqref{tech.3}, we have that
  \begin{equation}\label{tech.5}
  |\xi_2|^\alpha\le g(\delta)|\xi_1|^\alpha\textrm{ with } g(\delta)=\frac{\alpha+1+B+\delta}{\alpha+1+B-\delta}\xrightarrow[\delta\to 0]{}1.
  \end{equation}
  This implies that
  \begin{align*}
    \mu_2(\mu_1+\mu_2)&\le (B+\delta)|\xi_2|^\alpha - (B-\delta)^{1/2}|\xi_2|^{\alpha/2}\mu_1\\
    &\le (B+\delta)|\xi_2|^\alpha-(B-\delta)|\xi_1\xi_2|^{\alpha/2}\\
    &\le f_1(\delta)|\xi_2|^\alpha
  \end{align*}
  with
  $$
  f_1(\delta) = B+\delta-\frac{B-\delta}{g(\delta)^{1/2}} \xrightarrow[\delta\to 0]{}0.
  $$
  On the other hand, using \eqref{tech.5} again we infer
  $$
  |\xi_1+\xi_2|^\alpha = \left| |\xi_1|-|\xi_2| \right|^\alpha \le f_2(\delta)|\xi_1|^\alpha\textrm{ with } f_2(\delta) = \left(g(\delta)^{1/\alpha}-1\right)^\alpha\xrightarrow[\delta\to 0]{}0
  $$
  and
  $$
  |\xi_2|^\alpha-|\xi_1|^\alpha\le |\xi_1|^\alpha \le f_3(\delta)|\xi_1|^\alpha\textrm{ with } f_3(\delta)=g(\delta)-1\xrightarrow[\delta\to 0]{}0.
  $$
  Estimate \eqref{tech.4} follows then by choosing $f=f_1+f_2+f_3$.
\end{proof}

\begin{proof}[Proof of Proposition \ref{prop.bilinear}]
 First we show part (1). We observe that
 \begin{equation}\label{prop.bilinear.1}
 I:= \int_{\R^3} (f_1\ast f_2)\cdot f_3 = \int_{\R^3} (\widetilde{f_1}\ast f_3)\cdot f_2 = \int_{\R^3} (\widetilde{f_2}\ast f_3)\cdot f_1
 \end{equation}
 where $\widetilde{f_i}(\tau,\zeta) = f_i(-\tau,-\zeta)$. Therefore we can always assume that $L_1=L_{min}$. Moreover, let us define $f_i^\sharp(\theta,\zeta) = f_i(\theta+\omega(\zeta),\zeta)$ for $i=1,2,3$.
 In view of the assumptions on $f_i$, the functions $f_i^\sharp$ are supported in the sets
$$
D_{H_i,\infty,L_i}^\sharp = \{(\theta,\xi,\mu)\in\R^3: (\xi,\mu)\in \Delta_{H_i} \textrm{ and } |\theta|\le L_i\}.
$$
We also note that $\|f_i\|_{L^2} = \|f_i^\sharp\|_{L^2}$. Then it follows that
\begin{equation}\label{prop.bilinear.2}
I = \int_{\R^6} f_1^\sharp(\theta_1,\zeta_1) f_2^\sharp(\theta_2,\zeta_2) f_3^\sharp(\theta_1+\theta_2+\Omega(\zeta_1,\zeta_2), \zeta_1+\zeta_2) d\theta_1 d\theta_2 d\zeta_1 d\zeta_2.
\end{equation}
For $i=1,2,3$, we define $F_i(\zeta) = \left(\int_\R f_i^\sharp(\theta,\zeta)^2d\theta\right)^{1/2}$. Thus  applying the Cauchy-Schwarz and Young inequalities in the $\theta$ variable we get
\begin{align}
  I &\lesssim \int_{\R^4} \|f_1^\sharp(\cdot, \zeta_1)\|_{L^1_{\theta_1}} F_2(\zeta_2) F_3(\zeta_1+\zeta_2) d\zeta_1 d\zeta_2 \nonumber \\
  &\lesssim L_1^{1/2} \int_{\R^4} F_1(\zeta_1) F_2(\zeta_2) F_3(\zeta_1+\zeta_2) d\zeta_1 d\zeta_2. \label{prop.bilinear.2.0}
\end{align}
Since $\|\zeta\|\lesssim h(\zeta)^{\frac 1{\alpha}+\frac 12}$, estimate (\ref{prop.bilinear.00}) is deduced from (\ref{prop.bilinear.2.0}) by applying the same arguments in the $\xi, \mu$ variables.

Next we turn to the proof of part (2).
From (\ref{prop.bilinear.1}), we may assume $H_{min}=H_2$ and $L_{max}\neq L_1$, so that $H_2\ll H_1\sim H_3$. It suffices to prove that if $g_i:\R^2\to \R_+$ are $L^2$ functions supported in $\Delta_{H_i}$ for $i=1,2$ and $g:\R^3\to \R_+$ is an $L^2$ function supported in $D_{H_3,\infty, L_3}^\sharp$, then
\begin{equation}\label{prop.bilinear.3}
J(g_1,g_2,g) := \int_{\R^4} g_1(\zeta_1) g_2(\zeta_2) g(\Omega(\zeta_1,\zeta_2), \zeta_1+\zeta_2) d\zeta_1 d\zeta_2
\end{equation}
satisfies
\begin{equation}\label{prop.bilinear.4}
J(g_1,g_2,g) \lesssim H_1^{-1/2} H_2^{1/4} \|g_1\|_{L^2} \|g_2\|_{L^2} \|g\|_{L^2}.
\end{equation}
Indeed, if estimate (\ref{prop.bilinear.4}) holds, let us define $g_i(\zeta_i) = f_i^\sharp(\theta_i,\zeta_i)$, $i=1,2$, and $g(\Omega,\zeta) = f_3^\sharp(\theta_1+\theta_2+\Omega, \zeta)$ for $\theta_1$ and $\theta_2$ fixed. Hence, we would deduce applying (\ref{prop.bilinear.4}) and the Cauchy-Schwarz inequality to (\ref{prop.bilinear.2}) that
\begin{align}
 I &\lesssim H_1^{-1/2}H_2^{1/4} \|f_3^\sharp\|_{L^2} \int_{\R^2} \|f_1^\sharp(\theta_1, \cdot)\|_{L^2_\zeta} \|f_2^\sharp(\theta_2,\cdot)\|_{L^2_\zeta} d\theta_1 d\theta_2 \nonumber \\
 & \lesssim H_1^{-1/2} H_2^{1/4} L_1^{1/2} L_2^{1/2} \|f_1^\sharp\|_{L^2} \|f_2^\sharp\|_{L^2} \|f_3^\sharp\|_{L^2} \label{prop.bilinear.5},
\end{align}
which implies (\ref{prop.bilinear.0}) and \eqref{prop.bilinear.001}. To prove estimate (\ref{prop.bilinear.4}), we apply twice the Cauchy-Schwarz inequality to get that
$$
J(g_1,g_2,g) \le \|g_1\|_{L^2} \|g_2\|_{L^2} \left( \int_{\Delta_{H_1}\times\Delta_{H_2}} g(\Omega(\zeta_1,\zeta_2), \zeta_1+\zeta_2)^2 d\zeta_1 d\zeta_2\right)^{1/2}.
$$
Then we change variables $(\zeta_1',\zeta_2') = (\zeta_1+\zeta_2, \zeta_2)$, so that
\begin{equation}\label{prop.bilinear.6}
J(g_1,g_2,g) \le \|g_1\|_{L^2} \|g_2\|_{L^2} \left( \int_{\Delta_{\sim H_1}\times\Delta_{H_2}} g(\Omega(\zeta_1'-\zeta_2',\zeta_2'), \zeta_1')^2 d\zeta_1' d\zeta_2'\right)^{1/2}.
\end{equation}
Making the change of variable $(\xi_1, \mu_1, \xi_2, \mu_2) = (\xi_1', \mu_1', \Omega(\zeta_1'-\zeta_2',\zeta_2'), \mu_2')$, and noting that the Jacobi determinant satisfies
$$
|\partial_{\xi_2'}\Omega(\zeta_1'-\zeta_2',\zeta_2')| = |h(\zeta_1'-\zeta_2') - h(\zeta_2')| \sim H_{max},
$$
we get
$$
J(g_1,g_2,g) \lesssim H_1^{-1/2} \|g_1\|_{L^2} \|g_2\|_{L^2} \left(\int_{\R^3\times [-cH_2^{1/2},cH_2^{1/2}]} g(\xi_2, \xi_1, \mu_1)^2 d\xi_1 d\mu_1 d\xi_2 d\mu_2\right)^{1/2},
$$
which lead to (\ref{prop.bilinear.4}) after integrating in $\mu_2$.

Now we show part (3) and assume that the functions $f_i^\sharp$ are supported in the sets
$$
D_{H_i,N_i,L_i}^\sharp = \{(\theta,\xi,\mu)\in\R^3: \xi\in I_{N_i},\ (\xi,\mu)\in \Delta_{H_i} \textrm{ and } |\theta|\le L_i\}.
$$
In order to simplify the notations, we will denote $\zeta_3 = \zeta_1+\zeta_2$. We split the integration domain in the following subsets:
\begin{align*}
  \mathcal{D}_1 &= \{(\theta_1, \zeta_1, \theta_2,\zeta_2) \in \R^6 : \forall i\in\{1,2,3\}, \mu_i^2\ll |\xi_i|^\alpha\},\\
  \mathcal{D}_2 &= \{(\theta_1, \zeta_1, \theta_2,\zeta_2) \in \R^6\setminus \mathcal{D}_1 : \min_{1\le i\le 3}|\xi_i\mu_i|\ll \max_{1\le i\le 3}|\xi_i\mu_i|\},\\
  \mathcal{D}_3 &= \R^6\setminus \bigcup_{j=1}^2 \mathcal{D}_j.
\end{align*}
Then, if we denote by $I^j$ the restriction of $I$ given by (\ref{prop.bilinear.2}) to the domain $\mathcal{D}_j$, we have that
$$
I = \sum_{j=1}^3 I_j.
$$

\noindent
\textit{Estimate for $I_1$}. From (\ref{prop.bilinear.1}) we may assume $L_{max}=L_3$. Since $H_{min}\sim H_{max}$, it follows that $N_{min}\sim N_{max}$ and
\begin{align}
  |\Omega(\zeta_1,\zeta_2)| &= (|\xi_1+\xi_2|^\alpha(\xi_1+\xi_2) - |\xi_1|^\alpha\xi_1 - |\xi_2|^\alpha\xi_2) + (\xi_1+\xi_2)(\mu_1+\mu_2)^2 - \xi_1\mu_1^2 - \xi_2\mu_2^2 \nonumber\\
  &=  (|\xi_1+\xi_2|^\alpha(\xi_1+\xi_2) - |\xi_1|^\alpha\xi_1 - |\xi_2|^\alpha\xi_2) + 2\mu_1\mu_2 (\xi_1+\xi_2) + \xi_1\mu_2^2+\xi_2\mu_1^2 \label{prop.bilinear.9.0} \\
  &\sim N_{max}^{\alpha+1} \nonumber
\end{align}
in the region $\mathcal{D}_1$. We infer that $I_1$ is non zero only for $L_3\gtrsim N_{max}^{\alpha+1}$ and it suffices to show that
\begin{equation}\label{prop.bilinear.10}
  I_1 \lesssim N_{min}^{1/2} H_{min}^{1/4} L_{med}^{1/2} \|f_1^\sharp\|_{L^2} \|f_2^\sharp\|_{L^2} \|f_3^\sharp\|_{L^2}
\end{equation}
Arguing as in \eqref{prop.bilinear.2.0} we obtain estimate (\ref{prop.bilinear.10}).\\

\noindent
\textit{Estimate for $I_2$.} By definition of $\mathcal{D}_2$, there exists $i\in\{1,2,3\}$ such that $|\xi_i|^\alpha\lesssim \mu_i^2$. It follows that for any $j\in \{1,2,3\}$, we have $|\xi_j|^\alpha\lesssim H_j\sim H_i\sim \mu_i^2$ and therefore $N_{max}^\alpha\lesssim \max_{1\le j\le 3}\mu_j^2$. Moreover observe that since $N_{max}\sim N_{med}$ and $\max_{1\le j\le 3}|\mu_j|\sim \textrm{med}_{1\le j\le 3}|\mu_j|$, it holds that
$$
\max_{1\le j\le 3} |\xi_j\mu_j| \sim \max_{1\le j\le 3}|\xi_j| \max_{1\le j\le 3} |\mu_j|.
$$
From (\ref{prop.bilinear.1}) we may always assume $\min_{1\le j\le 3}|\xi_j\mu_j| = |\xi_1\mu_1|$ and $\max_{1\le j\le 3}|\xi_j\mu_j| = |\xi_2\mu_2|$. We deduce that in $\mathcal{D}_2$, it holds
$$
  |\partial_{\mu_2'}\Omega(\zeta_1'-\zeta_2',\zeta_2')| =  2|\xi_1\mu_1-\xi_2\mu_2|\gtrsim N_{max}^{1+\frac \alpha 2}
$$
where $(\zeta_1',\zeta_2') = (\zeta_1+\zeta_2,\zeta_2)$.
Changing the variable $(\xi_1,\mu_1,\xi_2,\mu_2) = (\xi_1',\mu_1',\xi_2', \Omega(\zeta_1'-\zeta_2',\zeta_2'))$ in (\ref{prop.bilinear.6}) we infer
$$
J_2(g_1,g_2,g) \lesssim N_{max}^{-\frac 12-\frac \alpha 4} \|g_1\|_{L^2} \|g_2\|_{L^2} \left(\int_{\R^2\times I_{N_2}\times\R} g(\mu_2, \xi_1, \mu_1)^2 d\xi_1 d\mu_1 d\xi_2 d\mu_2\right)^{1/2},
$$
where $J_2$ is the restriction of the integral $J$ defined by \eqref{prop.bilinear.3} to the domain $\mathcal{D}_2$.
This leads to
$$
I_2 \lesssim N_{max}^{-\alpha/4} L_{med}^{1/2}L_{max}^{1/2} \|f_1\|_{L^2} \|f_2\|_{L^2} \|f_3\|_{L^2},
$$
which is acceptable since $N_{max}^{\alpha/4}\lesssim H_{max}^{1/4}\sim H_{min}^{1/4}$.\\

\noindent
\textit{Estimate for $I_3$.} First we notice that in $\mathcal{D}_3$, we have
 $$
 N_{min}^\alpha\sim N_{max}^\alpha \lesssim \min_{1\le i\le 3}\mu_i^2\sim \max_{1\le i\le 3}\mu_i^2.
 $$
 Let $0<\delta \ll 1$ be a small positive number such that $f(\delta)=\frac{1}{1000}$ where $f$ is defined in Lemma \ref{tech}.  We split again the integration domain $\mathcal{D}_3$ in the following subsets:
\begin{align*}
  \mathcal{D}_3^1 &= \{(\theta_1, \zeta_1, \theta_2,\zeta_2) \in \mathcal{D}_3: \zeta_1,\zeta_2\in A_\delta\},\\
  \mathcal{D}_3^2 &= \{(\theta_1, \zeta_1, \theta_2,\zeta_2) \in \mathcal{D}_3: \zeta_2,\zeta_3\in A_\delta\},\\
  \mathcal{D}_3^3 &= \{(\theta_1, \zeta_1, \theta_2,\zeta_2) \in \mathcal{D}_3: \zeta_1,\zeta_3\in A_\delta\},\\
  \mathcal{D}_3^4 &= \mathcal{D}_3\setminus \bigcup_{j=1}^3 \mathcal{D}_3^j.
\end{align*}
Then, if we denote by $I_3^j$ the restriction of $I_3$ to the domain $\mathcal{D}_3^j$, we have that
$$
I_3 = \sum_{j=1}^4 I_3^j.
$$

\noindent
\textit{Estimate for $I_3^1$.} We consider the following subcases.
\begin{enumerate}
  \item \textit{Case $\{\xi_1\xi_2>0 \textrm{ and } \mu_1\mu_2>0\}$.} We define
  $$
  \mathcal{D}_3^{1,1} = \{(\theta_1, \zeta_1, \theta_2,\zeta_2) \in \mathcal{D}_3^1: \xi_1\xi_2>0 \textrm{ and } \mu_1\mu_2>0\}
  $$
  and denote by $I_3^{1,1}$ the restriction of $I_3^1$ to the domain $\mathcal{D}_3^{1,1}$. We observe from (\ref{prop.bilinear.9.0}) that
  $$
  L_{max}\gtrsim |\Omega(\zeta_1,\zeta_2)| \gtrsim N_{max}^{\alpha+1}
  $$
  in the region $\mathcal{D}_3^{1,1}$. Therefore, it follows arguing exactly as in (\ref{prop.bilinear.10}) that
  \begin{equation}
    I_3^{1,1} \lesssim N_{max}^{-\alpha/2} H_{min}^{1/4} L_{med}^{1/2} L_{max}^{1/2} \|f_1\|_{L^2}\|f_2\|_{L^2} \|f_3\|_{L^2}.
  \end{equation}
  \item \textit{Case $\{\xi_1\xi_2>0 \textrm{ and } \mu_1\mu_2<0\}$ or $\{\xi_1\xi_2<0 \textrm{ and } \mu_1\mu_2>0\}$.}
  We define
  $$
  \mathcal{D}_3^{1,2} = \{(\theta_1, \zeta_1, \theta_2,\zeta_2) \in \mathcal{D}_3^1: \xi_1\xi_2>0, \mu_1\mu_2<0 \textrm{ or } \xi_1\xi_2<0, \mu_1\mu_2>0\}
  $$
  and denote by $I_3^{1,2}$ the restriction of $I_3^1$ to the domain $\mathcal{D}_3^{1,2}$. Observe that
  $$
  |\partial_{\mu_2'}\Omega(\zeta_1'-\zeta_2',\zeta_2')| = 2 |\xi_1\mu_1-\xi_2\mu_2|\gtrsim N_{max}^{1+\alpha/2}
  $$
  in the region $\mathcal{D}_3^{1,2}$, where $(\zeta_1',\zeta_2') = (\zeta_1+\zeta_2,\zeta_2)$. Thus, arguing as in the proof of (\ref{prop.bilinear.4}), we get that the restriction of $J$ to $\mathcal{D}_3^{1,2}$ satisfies
  $$
  J_3^{1,2}(g_1,g_2,g) \lesssim N_{max}^{-\alpha/4}\|g_1\|_{L^2} \|g_2\|_{L^2} \|g\|_{L^2},
  $$
  which leads to
  \begin{equation}
    I_3^{1,2} \lesssim N_{max}^{-\alpha/2} H_{min}^{1/4} L_{med}^{1/2} L_{max}^{1/2} \|f_1\|_{L^2}\|f_2\|_{L^2} \|f_3\|_{L^2},
  \end{equation}
  since $N_{max}^{\alpha/4} \lesssim H_{max}^{1/4}\sim H_{min}^{1/4}$.
  \item \textit{Case $\{\xi_1\xi_2<0 \textrm{ and } \mu_1\mu_2<0\}$.}
  We define
  $$
  \mathcal{D}_3^{1,3} = \{(\theta_1, \zeta_1, \theta_2,\zeta_2) \in \mathcal{D}_3^1: \xi_1\xi_2<0 \textrm{ and } \mu_1\mu_2<0\}
  $$
  and denote by $I_3^{1,3}$ the restriction of $I_3^1$ to the domain $\mathcal{D}_3^{1,3}$. We observe due to the frequency localization that there exists some $0<\gamma \ll 1$ such that
  \begin{equation}\label{prop.bilinear.12}
    |h(\zeta_1)-h(\zeta_2)| \ge \gamma \max(h(\zeta_1), h(\zeta_2))
  \end{equation}
  in $\mathcal{D}_3^{1,3}$. Indeed, if estimate (\ref{prop.bilinear.12}) does not hold for all $0<\gamma\le \frac{1}{1000}$, then estimate (\ref{tech.1}) with $f(\delta)=\frac{1}{1000}$ would imply that
  $$
  h(\zeta_3)\le \frac{1}{500} \max(h(\zeta_1), h(\zeta_2))
  $$
  which would be a contradiction since $H_{min}\sim H_{max}$.
  Thus we deduce from (\ref{prop.bilinear.12}) that
  $$
  |\partial_{\xi_2'}\Omega(\zeta_1'-\zeta_2',\zeta_2')| = |h(\zeta_1)-h(\zeta_2)|\gtrsim H_{max}
  $$
  in the region $\mathcal{D}_3^{1,3}$, where $(\zeta_1',\zeta_2') = (\zeta_1+\zeta_2,\zeta_2)$. We can then reapply the arguments in the proof of (\ref{prop.bilinear.4}) to show that
  \begin{equation}
    I_3^{1,3} \lesssim N_{max}^{-\alpha/2} H_{min}^{1/4} L_{med}^{1/2} L_{max}^{1/2} \|f_1\|_{L^2}\|f_2\|_{L^2} \|f_3\|_{L^2}
  \end{equation}
\end{enumerate}

\noindent
\textit{Estimate for $I_3^2$ and $I_3^3$.} The estimates for these terms follow the same lines as for $I_3^1$.

\noindent
\textit{Estimate for $I_3^4$.} Without loss of generality, we can assume that $\zeta_1,\zeta_2 \in \R^2\setminus A_{\delta}$. Then we may take advantage of the improved Strichartz estimates derived in Section \ref{stri-sec}. We deduce from Plancherel's identity and H\"older's inequality that
$$
I_3^4 \lesssim \|f_3\|_{L^2} \|(1_{\R^2\setminus A_\delta}f_1)\ast (1_{\R^2\setminus A_\delta}f_2)\|_{L^2} \lesssim \|f_3\|_{L^2} \|P_{A_\delta^c}\F^{-1}(f_1)\|_{L^4} \|P_{A_\delta^c}\F^{-1}(f_2)\|_{L^4}.
$$
We conclude from Corollary \ref{l4} that
$$
I_3^4 \lesssim N_{max}^{(-\alpha/4)+} L_{med}^{1/2}L_{max}^{1/2} \|f_1\|_{L^2}\|f_2\|_{L^2} \|f_3\|_{L^2},
$$
which is acceptable since $N_{max}^{\alpha/4}\lesssim H_{min}^{1/4}$.
\end{proof}

As a consequence of Proposition \ref{prop.bilinear}, we have the following $L^2$ bilinear estimates.
\begin{corollary}\label{bilinear}
Assume that $H_i, N_i, L_i\in\D$ are dyadic numbers and $f_i:\R^3\to\R_+$ are $L^2$ functions for $i=1,2,3$.
\begin{enumerate}
\item If $f_i$  are supported in $D_{H_i, \infty,L_i}$ for $i=1,2,3$, then
\begin{equation}\label{bilinear.0}
\|1_{D_{H_3, \infty,L_3}}(f_1\ast f_2)\|_{L^2} \lesssim H_{min}^{\frac 1{2\alpha}+\frac 14} L_{min}^{1/2} \|f_1\|_{L^2}\|f_2\|_{L^2}.
\end{equation}
\item Let us suppose that $H_{min}\ll H_{max}$ and $f_i$  are supported in $D_{H_i, \infty,L_i}$ for $i=1,2,3$. If $(H_i,L_i)=(H_{min},L_{max})$ for some $i\in\{1,2,3\}$, then
\begin{equation}\label{bilinear.1}
\|1_{D_{H_3, \infty,L_3}}(f_1\ast f_2)\|_{L^2} \lesssim H_{max}^{-1/2} H_{min}^{1/4} L_{min}^{1/2}L_{max}^{1/2} \|f_1\|_{L^2}\|f_2\|_{L^2}.
\end{equation}
Otherwise, we have
\begin{equation}\label{bilinear.01}
\|1_{D_{H_3, \infty,L_3}}(f_1\ast f_2)\|_{L^2} \lesssim H_{max}^{-1/2} H_{min}^{1/4} L_{min}^{1/2}L_{med}^{1/2} \|f_1\|_{L^2}\|f_2\|_{L^2}.
\end{equation}
\item If $H_{min}\sim H_{max}$ and $f_i$ are  supported in $D_{H_i, N_i,L_i}$ for $i=1,2,3$, then
\begin{equation}\label{bilinear.2}
\|1_{D_{H_3, N_3,L_3}}(f_1\ast f_2)\|_{L^2} \lesssim N_{max}^{-\alpha/2} H_{min}^{(1/4)+} L_{med}^{1/2}L_{max}^{1/2} \|f_1\|_{L^2}\|f_2\|_{L^2}.
\end{equation}
\end{enumerate}
\end{corollary}
\begin{proof}
Corollary \ref{bilinear} follows directly from Proposition \ref{prop.bilinear} by using a duality argument.
\end{proof}

\section{Short time bilinear estimates}\label{sec-bil}

\begin{proposition}\label{stbe}
\begin{enumerate}
\item If $s>1/4$, $T\in (0,1]$ and $u,v\in F^s(T)$, then
\begin{equation}\label{stbe.1}
\|\partial_x(uv)\|_{\N^s(T)} \lesssim \|u\|_{F^s(T)} \|v\|_{F^{(1/4)+}(T)} + \|u\|_{F^{(1/4)+}(T)}\|v\|_{F^s(T)}.
\end{equation}
\item If $s>1/4$, $T\in (0,1]$, $u\in F^0(T)$ and $v\in F^s(T)$, then
\begin{equation}\label{stbe.01}
\|\partial_x(uv)\|_{\N^0(T)} \lesssim \|u\|_{F^0(T)} \|v\|_{F^{s}(T)}.
\end{equation}
\end{enumerate}
\end{proposition}

We split the proof of Proposition \ref{stbe} into several technical lemmas.

\begin{lemma}[$low \times high \to high$]\label{lhh}
Assume that $H, H_1, H_2\in\D$ satisfy $H_1\ll H\sim H_2$. Then,
\begin{equation}\label{lhh.0}
\|P_H \partial_x(u_{H_1} v_{H_2})\|_{\N_H} \lesssim H_1^{1/4}\|u_{H_1}\|_{F_{H_1}} \|v_{H_2}\|_{F_{H_2}},
\end{equation}
for all $u_{H_1}\in F_{H_1}$ and $v_{H_2}\in F_{H_2}$.
\end{lemma}
\begin{proof}
First observe from the definition of $\N_H$ in (\ref{def.FN}) that
\begin{equation}\label{lhh.1}
\|P_H \partial_x (u_{H_1}v_{H_2})\|_{\N_H} \lesssim \sup_{t_H\in\R}\|(\tau-\omega(\zeta)+iH^{\beta})^{-1} H^{1/\alpha} 1_{\Delta_H}\cdot f_{H_1} \ast g_{H_2}\|_{X_H},
\end{equation}
where
\begin{align*}
f_{H_1} &= |\F(\varphi_1(H^{\beta}(\cdot-t_H))u_{H_1})| \;, \\
g_{H_2} &= |\F(\widetilde{\varphi}_1(H^{\beta}(\cdot -t_H))v_{H_2})|.
\end{align*}
Now we set
\begin{align*}
f_{H_1,\lfloor H^{\beta}\rfloor}(\tau,\zeta) &= \varphi_{\le \lfloor H^{\beta}\rfloor}(\tau-\omega(\zeta))f_{H_1}(\tau,\zeta), \\
f_{H_1,L}(\tau,\zeta) &= \varphi_L(\tau-\omega(\zeta)) f_{H_1}(\tau,\zeta),
\end{align*}
for $L>\lfloor H^{\beta}\rfloor$ and we define similarly $g_{H_2,L}$ for $L\ge \lfloor H^\beta\rfloor$. Thus we deduce from (\ref{lhh.1}) and the definition of $X_H$ that
\begin{equation}\label{lhh.2}
\|P_H \partial_x (u_{H_1}v_{H_2})\|_{\N_H} \lesssim \sup_{t_H\in\R} H^{1/\alpha} \sum_{L,L_1,L_2\ge \lfloor H^{\beta}\rfloor} L^{-1/2} \|1_{D_{H,\infty,L}}\cdot f_{H_1,L_1}\ast g_{H_2,L_2}\|_{L^2},
\end{equation}
where $D_{H,\infty,L}$ is defined in (\ref{DNinfL}). Here we use that since $|(\tau-\omega(\zeta)+iH^{\beta})^{-1}|\le H^{-\beta}$, the sum for $L<\lfloor H^{\beta}\rfloor$ appearing implicitly on the RHS of (\ref{lhh.1}) is controlled by the term corresponding to $L=\lfloor H^{\beta}\rfloor$ on the RHS of (\ref{lhh.2}). Therefore, according to Corollary \ref{cor-fond} and estimate (\ref{lhh.2}) it suffices to prove that
\begin{multline}\label{lhh.3}
H^{1/\alpha}\sum_{L\ge \lfloor H^{\beta}\rfloor} L^{-1/2} \|1_{D_{H,\infty,L}} \cdot f_{H_1,L_1}\ast g_{H_2,L_2}\|_{L^2} \\
\lesssim H_1^{1/4}L_1^{1/2} \|f_{H_1,L_1}\|_{L^2} L_2^{1/2}\|g_{H_2,L_2}\|_{L^2}
\end{multline}
with $L_1,L_2\ge \lfloor H^{\beta}\rfloor$. Using that $\frac 1\alpha-\frac \beta 2-\frac 12=0$, this is a consequence of estimates (\ref{bilinear.1})-\eqref{bilinear.01}.
\end{proof}

\begin{lemma}[$high \times high \to high$]
Assume that $H, H_1, H_2\in\D$ satisfy $H\sim H_1\sim H_2\gg 1$. Then,
\begin{equation}\label{hhh.0}
\|P_H \partial_x(u_{H_1} v_{H_2})\|_{\N_H} \lesssim H^{(1/4)+} \|u_{H_1}\|_{F_{H_1}} \|v_{H_2}\|_{F_{H_2}},
\end{equation}
for all $u_{H_1}\in F_{H_1}$ and $v_{H_2}\in F_{H_2}$.
\end{lemma}

\begin{proof}
  Arguing as in the proof of Lemma \ref{lhh}, it is enough to prove that
  \begin{multline}\label{hhh.1}
N\sum_{L\ge \lfloor H^{\beta}\rfloor} L^{-1/2} \|1_{D_{H,N,L}} \cdot f_{H_1,N_1,L_1}\ast g_{H_2,N_2,L_2}\|_{L^2} \\
\lesssim H^{(1/4)+}L_1^{1/2} \|f_{H_1,N_1,L_1}\|_{L^2} L_2^{1/2}\|g_{H_2,N_2,L_2}\|_{L^2}
\end{multline}
where $f_{H_1,N_1,L_1}$ and $g_{H_2,N_2,L_2}$ are localized in $D_{H_i,N_i,L_i}$, with $L,L_1,L_2\ge \lfloor H^\beta\rfloor$ and $N,N_1,N_2\lesssim H^{1/\alpha}$. Observe that the sums over $N,N_1,N_2$ are easily controlled by $\log(H^{1/\alpha})\lesssim H^{0+}$. Using that $1-\frac 1\alpha\ge 0$ and $\frac 1\alpha-\frac \beta 2-\frac 12=0$, this is a consequence of estimate \eqref{bilinear.2} in the case $L=L_{min}$ or $L_{med}\sim L_{max}$. Otherwise, we have $L_{max}\sim |\Omega|\lesssim H^{1+\frac 1\alpha}$ so that the sum over $L$ is bounded by $H^{0+}$ and \eqref{hhh.1} still holds.
\end{proof}

\begin{lemma}[$high \times high \to low$]
Assume that $H, H_1, H_2\in\D$ satisfy $H\ll H_1\sim H_2$. Then,
\begin{equation}\label{hhl.0}
\|P_H \partial_x(u_{H_1} v_{H_2})\|_{\N_H} \lesssim H^{\frac 54-\frac 1\alpha} H_1^{(\frac 1\alpha-1)+} \|u_{H_1}\|_{F_{H_1}} \|v_{H_2}\|_{F_{H_2}},
\end{equation}
for all $u_{H_1}\in F_{H_1}$ and $v_{H_2}\in F_{H_2}$.
\end{lemma}

\begin{proof}
  Let $\gamma:\R\to [0,1]$ be a smooth function supported in $[-1,1]$ with the property that
  $$
  \sum_{m\in\Z}\gamma^2(x-m)=1,\ \forall x\in\R.
  $$
  We observe from the definition of $\N_H$ in  \eqref{def.NH} that
  \begin{multline}\label{hhl.1}
    \|P_H \partial_x(u_{H_1} v_{H_2})\|_{\N_H} \\
    \lesssim H^{1/\alpha}\sup_{t_H\in\R} \left\|(\tau-\omega(\zeta)+iH^\beta)^{-1} 1_{\Delta_H} \sum_{|m|\lesssim (H_1/H)^\beta}f_{H_1}^m\ast g_{H_2}^m \right\|_{X_H},
  \end{multline}
  where
  $$
  f_{H_1}^m = |\F(\varphi_1(H^\beta(\cdot-t_H))\gamma(H_1^\beta(\cdot-t_H)-m)u_{H_1})|,
  $$
  and
  $$
  g_{H_2}^m = |\F(\widetilde{\varphi}_1(H^\beta(\cdot-t_H))\gamma(H_1^\beta(\cdot-t_H)-m)v_{H_2})|.
  $$
  Now, we set
  \begin{align*}
  f_{H_1,\lfloor H_1^{\beta}\rfloor}^m(\tau,\zeta) &= \varphi_{\le \lfloor H_1^{\beta}\rfloor}(\tau-\omega(\zeta))f_{H_1}^m(\tau,\zeta), \\
  f_{H_1,L}^m(\tau,\zeta) &= \varphi_L(\tau-\omega(\zeta)) f_{H_1}^m(\tau,\zeta),
  \end{align*}
  for $L>\lfloor H_1^{\beta}\rfloor$ and we define similarly $g_{H_2,L}^m$ for $L\ge \lfloor H_1^\beta\rfloor$. Thus we deduce from \eqref{lhh.1} and the definition of $X_H$ that
  \begin{multline}\label{hhl.2}
  \|P_H \partial_x(u_{H_1} v_{H_2})\|_{\N_H} \\
    \lesssim H^{1/\alpha} \sup_{t_H\in\R, m\in\Z} H_1^\beta H^{-\beta} \sum_{L\in\D} \sum_{L_1,L_2\ge \lfloor H_1^\beta\rfloor} L^{-1/2} \|1_{D_{H,\infty,L}}\cdot f_{H_1,L_1}^m\ast g_{H_2,L_2}^m\|_{L^2}.
  \end{multline}
  Therefore, according to Lemma \ref{fond} and estimate \eqref{hhl.2} it suffices to prove that
  \begin{multline}\label{hhl.3}
    H^{\frac 1\alpha-\beta}H_1^\beta \sum_{L\in\D} L^{-1/2} \|1_{D_{H,\infty,L}}\cdot f_{H_1,L_1}^m\ast g_{H_2,L_2}^m\|_{L^2} \\
    \lesssim H^{\frac 54-\frac 1\alpha} H_1^{(\frac 1\alpha-1)+} L_1^{1/2} \|f_{H_1,L_1}^m\|_{L^2} L_2^{1/2}\|g_{H_2,L_2}^m\|_{L^2},
  \end{multline}
  with $L_1,L_2\ge \lfloor H_1^\beta\rfloor$ in order to prove \eqref{hhl.0}. As in the proof of Lemma \ref{lhh}, estimate \eqref{hhl.3} follows from \eqref{bilinear.1}-\eqref{bilinear.01} and the fact that $L_{max}\sim \max(L_{med}, |\Omega|)$.
\end{proof}

\begin{lemma}[$low \times low \to low$]
Assume that $H, H_1, H_2\in\D$ satisfy $H, H_1, H_2\lesssim 1$. Then,
\begin{equation}\label{lll.0}
\|P_H \partial_x(u_{H_1} v_{H_2})\|_{\N_H} \lesssim \|u_{H_1}\|_{F_{H_1}} \|v_{H_2}\|_{F_{H_2}},
\end{equation}
for all $u_{H_1}\in F_{H_1}$ and $v_{H_2}\in F_{H_2}$.
\end{lemma}

\begin{proof}
  Arguing as in the proof of Lemma \ref{lhh}, it is enough to prove that
  \begin{equation}\label{lll.1}
\sum_{L\in\D} L^{-1/2} \|1_{D_{H,\infty,L}} \cdot f_{H_1,L_1}\ast g_{H_2,L_2}\|_{L^2}
\lesssim L_1^{1/2} \|f_{H_1,L_1}\|_{L^2} L_2^{1/2}\|g_{H_2,L_2}\|_{L^2}
\end{equation}
where $f_{H_1,L_1}$ and $g_{H_2,L_2}$ are localized in $D_{H_i,\infty,L_i}$, with $L_1,L_2\in\D$, which is a direct consequence of estimate \eqref{bilinear.0}.
\end{proof}

\begin{proof}[Proof of Proposition \ref{stbe}]
We only prove part (1) since the proof of estimate \eqref{stbe.01} follows the same lines.
We choose two extensions $\widetilde{u}$ and $\widetilde{v}$ of $u$ and $v$ satisfying
\begin{equation}\label{stbe.2}
\|\widetilde{u}\|_{F^s}\le 2\|u\|_{F^s(T)} \textrm{ and } \|\widetilde{v}\|_{F^s}\le 2\|v\|_{F^s(T)}.
\end{equation}
We have from the definition of $\N^s(T)$ and Minkowski inequality that
$$
\|\partial_x(uv)\|_{\N^s(T)} \lesssim \left(\sum_{H} H^{2s}\left(\sum_{H_1,H_2} \|P_H\partial_x(\widetilde{u}_{H_1} \widetilde{v}_{H_2})\|_{\N_H}\right)^2\right)^{1/2}.
$$
Let us denote
\begin{align*}
A_1 &= \{ (H_1, H_2)\in\D^2 : H\ll H_1\sim H_2 \}, \\
A_2 &= \{ (H_1, H_2)\in\D^2 : H_1\ll H\sim H_2 \}, \\
A_3 &= \{ (H_1, H_2)\in\D^2 : H_2\ll H\sim H_1 \}, \\
A_4 &= \{ (H_1, H_2)\in\D^2 : H\sim H_1\sim H_2\gg 1 \}, \\
A_5 &= \{ (H_1, H_2)\in\D^2 : H, H_1, H_2\lesssim 1 \}.
\end{align*}
Due to the frequency localization, we have
\begin{align}
\|\partial_x(uv)\|_{\N^s(T)} &\lesssim \sum_{j=1}^5\Big(\sum_{H\in\D} H^{2s}\Big(\sum_{(H_1,H_2)\in A_j}  \|P_H\partial_x(\widetilde{u}_{H_1} \widetilde{v}_{H_2})\|_{\N_H}\Big)^2\Big)^{1/2} \nonumber \\
&:=  \sum_{j=1}^5 S_j. \label{stbe.3}
\end{align}
To handle the sum $S_1$, we use estimate (\ref{hhl.0}) to obtain that
\begin{equation}\label{stbe.4}
S_1\lesssim \Big( \sum_{H\in\D} H^{2s} \Big(\sum_{H_1\gg H} H_1^{(1/4)+} \|\widetilde{u}_{H_1}\|_{F_{H_1}} \|\widetilde{v}_{H_1}\|_{F_{H_1}}\Big)^2\Big)^{1/2}
\lesssim \|\widetilde{u}\|_{F^{(1/4)+}} \|\widetilde{v}\|_{F^s}.
\end{equation}
Estimate \eqref{lhh.0} leads to
\begin{equation}\label{stbe.5}
S_2\lesssim \Big( \sum_{H\in\D} H^{2s} \Big(\sum_{H_1\ll H} H_1^{1/4} \|\widetilde{u}_{H_1}\|_{F_{H_1}} \|\widetilde{v}_{H}\|_{F_{H}}\Big)^2\Big)^{1/2}
\lesssim \|\widetilde{u}\|_{F^{(1/4)+}} \|\widetilde{v}\|_{F^s}.
\end{equation}
Similarly we deduce by symmetry that
\begin{equation}\label{stbe.6}
S_3 \lesssim \|\widetilde{u}\|_{F^s} \|\widetilde{v}\|_{F^{(1/4)+}}
\end{equation}
Next it follows from estimate (\ref{hhh.0}) and Cauchy-Schwarz inequality that
\begin{equation}\label{stbe.7}
S_4 \lesssim \left(\sum_{H\in\D} H^{2s} H^{(1/2)+} \|\widetilde{u}_{H}\|_{F_{H}}^2 \|\widetilde{v}_{H}\|_{F_{H}}^2\right)^{1/2}
\lesssim \|\widetilde{u}\|_{F^{s}} \|\widetilde{v}\|_{F^{(1/4)+}}.
\end{equation}
Finally it is clear from estimate \eqref{lll.0} that
\begin{equation}\label{stbe.8}
  S_5\lesssim \|\widetilde{u}\|_{F^0} \|\widetilde{v}\|_{F^0}.
\end{equation}
Therefore we conclude the proof of \eqref{stbe.1} gathering \eqref{stbe.3}-\eqref{stbe.8}.
\end{proof}

\section{Energy estimates}\label{sec-energy}
The aim of this section is to derive energy estimates for the solutions of \eqref{bozk} and the solutions of the equation satisfied by the difference of two solutions of \eqref{bozk}.
In order to simplify the notations, we will instead derive energy estimates on the solutions $v$ of the more general equation
\begin{equation}\label{eq-u0}
  \partial_tv - D_x^\alpha\partial_xv + \partial_{xyy}v = c_1\partial_x(uv),
\end{equation}
where $u$ solves
\begin{equation}\label{eq-u1}
  \partial_tu - D_x^\alpha\partial_xu + \partial_{xyy}u = c_2\partial_x(u_1u_2).
\end{equation}
Here we assume $c_1,c_2\in\R^\ast$ and that all the functions $u,v,u_1,u_2$ are real-valued.

Let us define our new energy by
\begin{equation}\label{def-EH}
  \mathcal{E}_H(v)(t) =  \|P_Hv(t)\|_{L^2(\R^2)}^2 + H^{-1}\int_{\R^2}\Pi_{\eta}(P_{\ll H}u(t), v(t)) P_Hv(t)
\end{equation}
for any $H\in\D\setminus\{1\}$ and where $\eta$ is a bounded function uniformly in $H$ that will be fixed later. Finally we set
\begin{equation}\label{def-EsT}
  E^s_T(v) = \|P_1v(0)\|_{L^2(\R^2)}^2 + \sum_{H\in\D\setminus\{1\}} \sup_{t_H\in [-T,T]} H^{2s}|\mathcal{E}_H(v)(t_H)|.
\end{equation}

Note that for the integral in \eqref{def-EH} to be non zero, the first occurrence of the function $v$ must be localized in $\Delta_{\sim H}$.

First, we show that if $s\ge 0$, the energy $E^s_T(v)$ is coercive.
\begin{lemma}
  Let $s\ge 0$, $0<T\le 1$ and $u,v,u_1,u_2\in B^s(T)$ be solutions of \eqref{eq-u0}-\eqref{eq-u1} on $[0,T]$. Then it holds that
  \begin{equation}\label{coer.0}
    \|v\|_{B^s(T)}^2 \lesssim E^s_T(v) + \|u\|_{B^0(T)} \|v\|_{B^s(T)}^2.
  \end{equation}
\end{lemma}

\begin{proof}
  We infer from \eqref{def-EsT}, the definition of $B^s(T)$ and the triangle inequality that
  \begin{equation}\label{coer.1}
    \|v\|_{B^s(T)}^2\lesssim E^s_T(v) + \sum_{H\in\D\setminus\{1\}} \sup_{t_H\in[-T,T]} H^{2s-1} \left|\int_{\R^2} (\Pi_{\eta}(P_{\ll H}u, v) P_Hv)(t_H)\right|
  \end{equation}
  Thanks to estimate \eqref{PiEst}, we have
  \begin{multline}\label{coer.2}
    H^{2s-1}\left|\int_{\R^2}(\Pi_{\eta}(P_{\ll H}u, v) P_Hv)(t_H)\right| \\
    \lesssim H^{2s-1}H^{\frac 1{2\alpha}+\frac 14} \|P_{\ll H}u(t_H)\|_{L^2} \|P_{\sim H}v(t_H)\|_{L^2} \|P_Hv(t_H)\|_{L^2}.
  \end{multline}
  Since $-1+\frac 1{2\alpha}+\frac 14\le 0$, we deduce estimate \eqref{coer.0} for $s\ge 0$ gathering \eqref{coer.1}-\eqref{coer.2} and using Cauchy-Schwarz.
\end{proof}

\begin{proposition}\label{ee}
Assume $s > s_\alpha$ and $T\in (0,1]$. Then if $u,v,u_1,u_2\in C([-T,T]; E^\infty)$ are solutions of \eqref{eq-u0}-\eqref{eq-u1}, we have that
\begin{multline}\label{ee.0}
E^s_T(v) \lesssim (1+\|u_0\|_{E^0})\|v_0\|_{E^s}^2 + \|u\|_{F^{s_\alpha+}(T)} \|v\|_{F^{s}(T)}^2+\|u\|_{F^{s+s_\alpha+}(T)}\|v\|_{F^0(T)} \|v\|_{F^s(T)}\\
+ (\|u\|_{B^{s}(T)}^2 + \|u_1\|_{B^{s}(T)}\|u_2\|_{B^{s}(T)}) \|v\|_{B^s(T)}^2.
\end{multline}
and
\begin{multline}\label{ee.00}
E^0_T(v) \lesssim (1+\|u_0\|_{E^0})\|v_0\|_{E^0}^2 + \|u\|_{F^{s_\alpha+}(T)} \|v\|_{F^{0}(T)}^2 \\
+ (\|u\|_{B^{s_\alpha+}(T)}^2 + \|u_1\|_{B^{s_\alpha+}(T)}\|u_2\|_{B^{s_\alpha+}(T)}) \|v\|_{B^0(T)}^2.
\end{multline}
Moreover in the case where $u=v$ it holds that
\begin{equation}\label{ee.000}
E^s_T(u) \lesssim (1+\|u_0\|_{E^0})\|u_0\|_{E^s}^2 + \|u\|_{F^{s_\alpha+}(T)} \|u\|_{F^{s}(T)}^2 + \|u\|_{B^{s_\alpha+}(T)}^2\|u\|_{B^s(T)}^2.
\end{equation}
\end{proposition}

The following result will be of constant use in the proof of Proposition \ref{ee}.
\begin{lemma}\label{estI}
  Assume that $T\in(0,1]$, $H_1,H_2,H_3\in \D$ and that $u_i\in F_{H_i}$ for $i=1,2,3$.
  \begin{enumerate}
    \item In the case $H_{min}\ll H_{max}$ it holds that
    \begin{equation}\label{estI.0}
      \left| \int_{[0,T]\times\R^2} \Pi_\eta(u_1,u_2)u_3 \right| \lesssim (H_{max}^{\frac 1\alpha-1}\vee H_{max}^{(-\frac 12)+}) H_{min}^{1/4} \prod_{i=1}^3 \|u_i\|_{F_{H_i}}.
    \end{equation}
    \item If $\F(u_i)$ are supported in $\R\times I_{N_i}\times\R$ for $i=1,2,3$ and $H_{min}\sim H_{max}$ then
    \begin{equation}\label{estI.1}
      \left| \int_{[0,T]\times\R^2} \Pi_\eta(u_1,u_2)u_3 \right| \lesssim N_{max}^{-\alpha/2} H_{min}^{(\frac 1\alpha-\frac 14)+} \prod_{i=1}^3 \|u_i\|_{F_{H_i}}.
    \end{equation}
  \end{enumerate}
\end{lemma}

\begin{remark}
 Observe that in the right-hand side of \eqref{estI.0}, we have $H_{max}^{\frac 1\alpha-1}\vee H_{max}^{(-\frac 12)+} = H_{max}^{\frac 1\alpha-1}$ as soon as $\alpha <2$. The lost of $H_{max}^{0+}$ in the particular case $\alpha=2$ is due to the localization in $[0,T]$.
\end{remark}

\begin{proof}
  From \eqref{symPi} we may always assume $H_1\le H_2\le H_3$. We first prove estimate \eqref{estI.0}. Let $\gamma:\R\to [0,1]$ be a smooth function supported in $[-1,1]$ with the property that
  $$
  \sum_{m\in\Z} \gamma^3(x-m)=1,\ \forall x\in\R.
  $$
  Then it follows that
  \begin{equation}\label{estI.2}
  \left| \int_{[0,T]\times\R^2} \Pi_\eta(u_1,u_2)u_3 \right| \lesssim \sum_{|m|\lesssim H_3^\beta} I_T^m
  \end{equation}
  with
  \begin{equation}\label{estI.3}
    I_T^m =
      \left| \int_{\R^3} \Pi_\eta \left(\gamma(H_3^\beta t-m)1_{[0,T]} u_1, \gamma(H_3^\beta t-m)1_{[0,T]} u_2\right) \gamma(H_3^\beta t-m)1_{[0,T]} u_3\right|.
  \end{equation}
  Now we observe that the sum on the right-hand side of \eqref{estI.2} is taken over the two disjoint sets
  $$
  \mathcal{A} = \{m\in\Z : \gamma(H_3^\beta t-m)1_{[0,T]} = \gamma(H_3^\beta t-m)\},
  $$
  and
  $$
  \mathcal{B} = \{ m\in\Z : \gamma(H_3^\beta t-m)1_{[0,T]} \neq \gamma(H_3^\beta t-m)\textrm{ and } \gamma(H_3^\beta t-m)1_{[0,T]}\neq 0\}.
  $$
  To deal with the sum over $\mathcal{A}$, we set
  $$
  f_{H_i,\lfloor H_3^\beta\rfloor}^m = \varphi_{\le \lfloor H_3^\beta\rfloor}(\tau-\omega(\zeta)) |\F(\gamma(H_3^\beta t-m)u_i)|,
  $$
  and
  $$
  f_{H_i,L}^m = \varphi_L(\tau-\omega(\zeta)) |\F(\gamma(H_3^\beta t-m)u_i)|,\ L>\lfloor H_3^\beta\rfloor,
  $$
  for each $m\in\mathcal{A}$ and $i\in\{1,2,3\}$. Therefore, we deduce by using Plancherel's identity and estimates \eqref{prop.bilinear.0}-\eqref{prop.bilinear.001} that
  \begin{align*}
    \sum_{m\in\mathcal{A}} I_T^m
    &\lesssim \sup_{m\in\mathcal{A}} H_3^\beta \|\eta\|_{L^\infty} \sum_{L_1,L_2,L_3\ge \lfloor H_3^\beta\rfloor} \int_{\R^3} (f_{H_1,L_1}^m \ast f_{H_2,L_2}^m)\cdot f_{H_3,L_3}^m \\
    &\lesssim \sup_{m\in\mathcal{A}} H_3^{\frac{\beta-1}2} H_1^{1/4} \prod_{i=1}^3 \sum_{L_i\ge \lfloor H_3^\beta\rfloor}L_i^{1/2} \|f_{H_i,L_i}^m\|_{L^2}.
  \end{align*}
  This implies together with Corollary \ref{cor-fond} that
  \begin{equation}\label{estI.4}
    \sum_{m\in\mathcal{A}} I_T^m
    \lesssim H_{3}^{\frac 1\alpha-1}H_{1}^{1/4} \prod_{i=1}^3 \|u_i\|_{F_{H_i}}.
    \end{equation}
  Now observe that $\#\mathcal{B}\lesssim 1$. We set
  $$
  g_{H_i,L}^m = \varphi_L(\tau-\omega(\zeta)) |\F(\gamma(H_3^\beta t-m)1_{[0,T]}u_i)|
  $$
  for $i\in\{1,2,3\}$, $L\in\D$ and $m\in\mathcal{B}$. Then, we deduce using again \eqref{prop.bilinear.0}-\eqref{prop.bilinear.001} as well as Lemma \ref{techX} that
  \begin{align}
    \sum_{m\in\mathcal{B}} I_T^m &\lesssim \sup_{m\in\mathcal{B}} \sum_{L_1,L_2,L_3\in\D} \int_{\R^3} (g_{H_1,L_1}^m\ast g_{H_2,L_2}^m)\cdot g_{H_3,L_3}^m \notag\\
    &\lesssim \sup_{m\in\mathcal{B}} H_3^{-1/2}H_1^{1/4} \sum_{L_1,L_2,L_3\in\D\atop L_{max}\sim\max(L_{med},|\Omega|)} L_{med}^{-1/2} \prod_{i=1}^3 \sup_{L_i\in\D} L_i^{1/2} \|g_{H_i,L_i}^m\|_{L^2}\notag \\
    &\lesssim H_3^{(-1/2)+} H_1^{1/4} \prod_{i=1}^3 \|u_i\|_{F_{H_i}}. \label{estI.5}
  \end{align}
  We deduce estimate \eqref{estI.0} gathering \eqref{estI.2}-\eqref{estI.5}. Finally, the proof of \eqref{estI.1} follows the same lines by using \eqref{prop.bilinear.01} instead of \eqref{prop.bilinear.0}-\eqref{prop.bilinear.001}. We also need to interpolate \eqref{prop.bilinear.01} with \eqref{prop.bilinear.00} to get
  $$
  \int_{\R^3} (f_1\ast f_2)\cdot f_3 \lesssim N_{max}^{-1/2} H_{min}^{\frac 14+(\frac 1{2\alpha}+\frac 34)\eps} L_{min}^{\eps/2}L_{med}^{\frac{1-\eps}2}L_{max}^{\frac{1-\eps}2}
  $$
  for $\eps\in(0,1)$. With this estimate in hand, we are able to control the contribution of the sum in the region $\mathcal{B}$.
\end{proof}

\begin{proof}[Proof of Proposition \ref{ee}] Let $v,u,u_1,u_2\in C([-T,T], E^\infty)$ be solutions to \eqref{eq-u0}-\eqref{eq-u1}. We choose some extensions $\widetilde{v}, \widetilde{u}, \widetilde{u_1}, \widetilde{u_2}$ of $v,u,u_1,u_2$ respectively on $\R^3$ satisfying $\|\widetilde{v}\|_{F^s}\lesssim \|v\|_{F^s(T)}$, $\|\widetilde{u}\|_{F^s}\lesssim \|u\|_{F^s(T)}$ and $\|\widetilde{u_i}\|_{F^s}\lesssim \|u_i\|_{F^s(T)}$ for $i=1,2$.

We fix $s>s_\alpha$ and set $\sigma\in\{0,s\}$.
 Then, for any $H\in \D\setminus\{1\}$, we differentiate $\mathcal{E}_H(v)$ with respect to $t$ and deduce using \eqref{eq-u0}-\eqref{eq-u1} as well as \eqref{PiPart} that
\begin{equation}\label{ee.1}
\frac{d}{dt}\mathcal{E}_H(v) = \mathcal{I}_H(v) + \mathcal{L}_H(v) + \mathcal{N}_H(v)
\end{equation}
with
$$
\mathcal{I}_H(v) = -2c_1 \int_{\R^2} P_H(uv)P_Hv_x,
$$
\begin{align*}
\mathcal{L}_H(v) &= -H^{-1} \int_{\R^2} \Pi_\eta(P_{\ll H}(-D^{\alpha}_x\partial_x+\partial_{xyy})u, v) P_Hv \\
&\quad - H^{-1}\int_{\R^2} \Pi_\eta(P_{\ll H}u, (-D^{\alpha}_x\partial_x+\partial_{xyy})v) P_Hv  \\
&\quad -H^{-1} \int_{\R^2} \Pi_\eta(P_{\ll H}u, v) P_H(-D^{\alpha}_x\partial_x+\partial_{xyy})v,
\end{align*}
and
\begin{align*}
\mathcal{N}_H(v) &= c_2H^{-1} \int_{\R^2} \Pi_\eta(P_{\ll H}\partial_x(u_1u_2), v)P_Hv \\
&\quad + c_1H^{-1}\int_{\R^2} \Pi_\eta(P_{\ll H}u, \partial_x(uv)) P_Hv\\
& \quad +c_1H^{-1}\int_{\R^2} \Pi_\eta(P_{\ll H}u, v) P_H\partial_x(uv)\\
&:= \mathcal{N}_H^1(v) + \mathcal{N}_H^2(v)+ \mathcal{N}_H^3(v).
\end{align*}
Now we fix $t_H\in [-T,T]$. Without loss of generality, we can assume that $0<t_H<T$. Therefore we obtain integrating \eqref{ee.1} between 0 and $t_H$ that
\begin{equation}\label{ee.2}
|\mathcal{E}_H(v)(t_H)| \le |\mathcal{E}_H(v)(0)| + \left| \int_0^{t_H} (\mathcal{I}_H(v) + \mathcal{L}_H(v) + \mathcal{N}_H(v)) dt\right|.
\end{equation}
Using H\"{o}lder and Bernstein inequalities, the first term in the right-hand side of \eqref{ee.2} is easily estimated by
\begin{equation}\label{ee.3}
\sum_{H\in \D\setminus\{1\}} H^{2\sigma}|\mathcal{E}_H(v)(0)| \lesssim (1+\|u_0\|_{E^0})\|v_0\|_{E^\sigma}^2.
\end{equation}
Next we estimate the second term in the right-hand side of \eqref{ee.2}.\\

\noindent
\textit{Estimates for the cubic terms.} By localization considerations, we obtain
\begin{align*}
\mathcal{I}_H(v) &= -2c_1\int_{\R^2}P_H(P_{\ll H}u v)P_Hv_x - 2c_1\int_{\R^2} P_H(uP_{\ll H}v)P_Hv_x\\
&\quad -2c_1\int_{\R^2} P_H(P_{\sim H}uP_{\sim H}v)P_Hv_x -2c_1 \sum_{H_1\gg H} \int_{\R^2} P_H(P_{H_1}u P_{\sim H_1}v)P_Hv_x\\
&:= \sum_{i=1}^4 \mathcal{I}_H^i(v).
\end{align*}
Note that in the case where $u=v$, we have $\mathcal{I}_H^1(v)=\mathcal{I}_H^2(v)$. Clearly we get by estimate \eqref{estI.1} that
\begin{align*}
  \left|\int_0^{t_H}\mathcal{I}_H^3(v)dt\right| &\lesssim \sum_{N_1,N_2,N_3\lesssim H^{1/\alpha}} H^{(\frac 1\alpha-\frac 14)+} N_3 N_{max}^{-\alpha/2} \|P_{\sim H}P_{N_1}^x\widetilde{u}\|_{F_H} \|P_{\sim H}P_{N_2}^x\widetilde{v}\|_{F_H} \|P_HP_{N_3}^x\widetilde{v}\|_{F_H}\\
  &\lesssim H^{s_\alpha+} \|P_{\sim H}\widetilde{u}\|_{F_H} \|P_{\sim H}\widetilde{v}\|_{F_H} \|P_H\widetilde{v}\|_{F_H},
\end{align*}
which combined with Cauchy-Schwarz inequality yields
\begin{equation}\label{ee.4}
  \sum_{H\in\D\setminus\{1\}} \sup_{t_H\in [0,T]} H^{2\sigma}\left|\int_0^{t_H}\mathcal{I}_H^3(v)dt\right|\lesssim \|u\|_{F^{s_\alpha+}(T)} \|v\|_{F^\sigma(T)}^2.
\end{equation}
Similarly, we get applying estimate \eqref{estI.0} that
$$
  \left|\int_0^{t_H}\mathcal{I}_H^4(v)dt\right|\lesssim  \sum_{H_1\gg H} H_1^{(\frac 1\alpha-1)+} H^{1/4} H^{1/\alpha} \|P_{H_1}\widetilde{u}\|_{F_{H_1}} \|P_{\sim H_1}\widetilde{v}\|_{F_{H_1}} \|P_H\widetilde{v}\|_{F_H}.
$$
From this and Cauchy-Schwarz inequality we infer
\begin{equation}\label{ee.5}
  \sum_{H\in\D\setminus\{1\}} \sup_{t_H\in [0,T]} H^{2\sigma} \left|\int_0^{t_H}\mathcal{I}_H^4(v)dt\right|\lesssim \|u\|_{F^{s_\alpha+}(T)} \|v\|_{F^\sigma(T)}^2.
\end{equation}
In the case $u\neq v$ we estimate $\mathcal{I}_H^2(v)$ thanks to Lemma \ref{estI} by
$$
\left|\int_0^{t_H}\mathcal{I}_H^2(v)dt\right| \lesssim \sum_{H_1\ll H} H^{s_\alpha+} \|P_{\sim H}u\|_{F_H}  \|P_{H_1}v\|_{F_H} \|P_Hv\|_{F_H}
$$
so that
\begin{equation}\label{ee.6}
\sum_{H\in\D\setminus\{1\}} \sup_{t_H\in [0,T]} H^{2\sigma} \left|\int_0^{t_H}\mathcal{I}_H^2(v)dt\right|\lesssim \|u\|_{F^{\sigma+s_\alpha+}(T)} \|v\|_{F^0(T)}  \|v\|_{F^\sigma(T)}.
\end{equation}
Therefore, it remains to estimate $\mathcal{I}_H^1(v)+\mathcal{L}_H(v)$ in the case $u\neq v$ and $2\mathcal{I}_H^1(v)+\mathcal{L}_H(v)$ when $u=v$.
Using a Taylor expansion of $\psi_H$ we may decompose $\mathcal{I}_H^1(v)$ as
\begin{align*}
\mathcal{I}_H^1(v) &= -2c_1\int_{\R^2} P_{\ll H}u P_HvP_Hv_x - 2c_1 H^{-1/\alpha} \int_{\R^2} \Pi_{\eta_1}(P_{\ll H}u_x, v) P_Hv_x\\
&\quad -2c_1 H^{-1} \int_{\R^2} \Pi_{\eta_2}(P_{\ll H} u_{yy}, v) P_Hv_x -2c_1 H^{-1}\int_{\R^2} \Pi_{\eta_3}(P_{\ll H}u_y, v_y) P_Hv_x\\
&:= \sum_{i=1}^4 \mathcal{I}^{1i}_H(v)
\end{align*}
where $\eta_i$, $i=1,2,3$ are bounded uniformly in $H$ and defined by
\begin{align*}
\eta_1(\zeta_1,\zeta_2) &= -i\alpha H^{\frac 1\alpha-1} \int_0^1 |\theta\xi_1+\xi_2|^{\alpha-1}\sgn(\theta\xi_1+\xi_2) \varphi'\left(\frac{|\theta\xi_1+\xi_2|^\alpha+(\theta\mu_1+\mu_2)^2}H\right)d\theta\\
\eta_2(\zeta_1,\zeta_2) &= -2  \int_0^1 \theta \varphi'\left(\frac{|\theta\xi_1+\xi_2|^\alpha+(\theta\mu_1+\mu_2)^2}H\right)d\theta\\
\eta_3(\zeta_1,\zeta_2) &= -2  \int_0^1  \varphi'\left(\frac{|\theta\xi_1+\xi_2|^\alpha+(\theta\mu_1+\mu_2)^2}H\right)d\theta
\end{align*}
To estimate the contribution of $\mathcal{I}_H^{11}(v)$, we integrate by parts and use \eqref{estI.0} to obtain
\begin{equation}\label{ee.7}
\left|\int_0^{t_H}\mathcal{I}_H^{11}(v)dt\right| \lesssim \sum_{H_1\ll H} (H^{\frac 1\alpha-1}\vee H^{(-\frac 12)+}) H_1^{\frac 1\alpha+\frac 14} \|P_{H_1}u\|_{F_{H_1}} \|P_Hv\|_{F_H}^2.
\end{equation}
Estimates for $\mathcal{I}_H^{12}(v)$ and $\mathcal{I}_H^{13}(v)$ are easily obtained thanks to \eqref{estI.0}:
\begin{multline}\label{ee.8}
\left|\int_0^{t_H} (\mathcal{I}_H^{12}(v) + \mathcal{I}_H^{13}(v))dt\right| \\
\lesssim \sum_{H_1\ll H} (H_1^{1/\alpha}+ H_1H^{\frac 1\alpha-1}) (H^{\frac 1\alpha-1}\vee H^{(-\frac 12)+}) H_1^{1/4} \|P_{H_1}u\|_{F_{H_1}} \|P_{\sim H}v\|_{F_H} \|P_Hv\|_{F_H}.
\end{multline}
Combining estimates \eqref{ee.7}-\eqref{ee.8} we infer
\begin{equation}\label{ee.9}
\sum_{H\in\D\setminus\{1\}} \sup_{t_H\in [0,T]} H^{2\sigma} \left|\sum_{i=1}^3 \int_0^{t_H} \mathcal{I}_H^{1i}(v)dt\right|\lesssim \|u\|_{F^{s_\alpha+}(T)} \|v\|_{F^\sigma(T)}  \|v\|_{F^\sigma(T)}.
\end{equation}
Note that due to the lack of derivative on the lowest frequencies term $P_{\ll H}u$, Lemma \ref{estI} does not permit to control the term $\mathcal{I}_H^{14}(v)$ without loosing a $H^{\frac 2\alpha-\frac 32}$ factor. This is why we modify the energy by adding the cubic term in \eqref{def-EH}. Let us rewrite $\mathcal{L}_H(v)$ as $\sum_{i=1}^3\mathcal{L}_H^i(v)$ with
$$
\mathcal{L}_H^1(v) = -H^{-1} \int_{\R^2} \Pi_\eta(P_{\ll H}(-D^{\alpha}_x\partial_x+\partial_{xyy})u, v) P_Hv,
$$
$$
\mathcal{L}_H^2(v) = H^{-1} \int_{\R^2} \left(\Pi_\eta(P_{\ll H}u, D^\alpha_x\partial_x v) P_Hv + \Pi_\eta(P_{\ll H}u, v) P_HD^\alpha_x\partial_xv\right),
$$
and
$$
\mathcal{L}_H^3(v) = -H^{-1} \int_{\R^2} \left(\Pi_\eta(P_{\ll H}u, v_{xyy}) P_Hv + \Pi_\eta(P_{\ll H}u, v) P_Hv_{xyy}\right).
$$
After a few integrations by parts, we obtain thanks to \eqref{PiPart} that
\begin{multline*}
  \mathcal{L}_H^3(v) = -2H^{-1} \int_{\R^2} \Pi_\eta(P_{\ll H}u_y, v_y) P_Hv_x - H^{-1}\int_{\R^2} \Pi_\eta(P_{\ll H}u_{yy}, v)P_Hv_x \\
  + H^{-1}\int_{\R^2} \Pi_\eta(P_{\ll H}u_x, v_{yy}) P_Hv.
\end{multline*}
Choosing $\eta = -\frac{1}{c_1}\eta_3$, a cancellation occurs and we get
\begin{align*}
\mathcal{I}_H^{14}(v) + \mathcal{L}_H^3(v) &= H^{-1}\int_{\R^2} \Pi_\eta(P_{\ll H}u_x, v_{yy}) P_Hv - H^{-1}\int_{\R^2} \Pi_\eta(P_{\ll H}u_{yy}, v)P_Hv_x\\
&:= \mathcal{L}_H^{31}(v) + \mathcal{L}_H^{32}(v).
\end{align*}
In the case $u=v$, it suffices to set $\eta = -\frac{1}{2c_1}\eta$ to obtain $2\mathcal{I}_H^{14}(v) + \mathcal{L}_H^3(v) = \mathcal{L}_H^{31}(v) + \mathcal{L}_H^{32}(v)$.
 Now we use estimate \eqref{estI.0} to bound the terms $\mathcal{L}_H^{31}(v)$, $\mathcal{L}_H^{32}$ as well as $\mathcal{L}_H^1(v)$. We get that
 \begin{multline*}
   \left| \int_0^{t_H}(\mathcal{L}_H^{31}(v) + \mathcal{L}_H^{32}(v) + \mathcal{L}_H^1(v))dt\right| \\
   \lesssim \sum_{H_1\ll H} H^{-1}(H_1^{1/\alpha}H + H_1H^{1/\alpha} + H_1^{\frac 1\alpha+1}) (H^{\frac 1\alpha-1}\vee H^{(-\frac 12)+}) H_1^{1/4} \|P_{H_1}u\|_{F_{H_1}} \|P_{\sim H}v\|_{F_H} \|P_Hv\|_{F_H}.
 \end{multline*}
 It follows that
 \begin{equation}\label{ee.10}
\sum_{H\in\D\setminus\{1\}} \sup_{t_H\in [0,T]} H^{2\sigma} \left|\int_0^{t_H}(\mathcal{L}_H^{31}(v) + \mathcal{L}_H^{32}(v) + \mathcal{L}_H^1(v))dt\right|\lesssim \|u\|_{F^{s_\alpha+}(T)} \|v\|_{F^\sigma(T)}  \|v\|_{F^\sigma(T)}.
\end{equation}
Finally to deal with $\mathcal{L}_H^2(v)$, we integrate by parts and use that
$$
|\xi_1+\xi_2|^\alpha - |\xi_2|^\alpha = \alpha\xi_1 \int_0^1 |\theta\xi_1+\xi_2|^{\alpha-1} \sgn(\theta\xi_1+\xi_2)d\theta.
$$
We deduce
\begin{align*}
  \mathcal{L}_H^2(v) &= -H^{-1}\int_{\R^2} \left(D^\alpha_x\Pi_\eta(P_{\ll H}u, v_x) - \Pi_\eta(P_{\ll H}u, D^\alpha_xv_x)\right)P_Hv \\
  &\quad -H^{-1} \int_{\R^2} \Pi_\eta(P_{\ll H}u_x, v)P_HD^\alpha_xv\\
  &= -H^{-1/\alpha} \int_{\R^2} \Pi_{\eta\widetilde{\eta}}(P_{\ll H}u_x, v_x)P_Hv -H^{-1} \int_{\R^2} \Pi_\eta(P_{\ll H}u_x, v)P_HD^\alpha_xv,
\end{align*}
with
$$
\widetilde{\eta}(\zeta_1,\zeta_2) = -i\alpha H^{\frac 1\alpha-1}\int_0^1 |\theta\xi_1+\xi_2|^{\alpha-1} \sgn(\theta\xi_1+\xi_2)d\theta.
$$
Noticing that $\widetilde{\eta}$ is bounded on $\Delta_{\ll H}\times\Delta_{\sim H}$ we easily get from Lemma \ref{estI} that
 \begin{equation}\label{ee.11}
\sum_{H\in\D\setminus\{1\}} \sup_{t_H\in [0,T]} H^{2\sigma} \left|\int_0^{t_H}\mathcal{L}_H^{2}(v)dt\right|\lesssim \|u\|_{F^{s_\alpha+}(T)} \|v\|_{F^\sigma(T)}  \|v\|_{F^\sigma(T)}.
\end{equation}
Gathering \eqref{ee.4}-\eqref{ee.11} we conclude
 \begin{multline}\label{ee.12}
\sum_{H\in\D\setminus\{1\}} \sup_{t_H\in [0,T]} H^{2\sigma} \left|\int_0^{t_H} (\mathcal{I}_H(v)+\mathcal{L}_H(v))dt\right|\\
\lesssim (\|u\|_{F^{s_\alpha+}(T)} \|v\|_{F^\sigma(T)} + \|u\|_{F^{\sigma+s_\alpha+}(T)} \|v\|_{F^0(T)}) \|v\|_{F^\sigma(T)}.
\end{multline}

\noindent
\textit{Estimates for the fourth order terms.}
We get using \eqref{PiEst} and H\"{o}lder inequality that
\begin{align*}
  |\mathcal{N}_H^1(v)| &\lesssim \sum_{H_1\ll H} H^{-1} H_1^{\frac 1\alpha} H_1^{\frac 1{2\alpha}+\frac 14} \|P_{H_1}(u_1u_2)\|_{L^2} \|P_{\sim H}v\|_{L^2} \|P_Hv\|_{L^2}\\
  &\lesssim \sum_{H_1\ll H} H_1^{\frac 3{2\alpha}-\frac 34} \left(\|P_{\ll H_1}u_1\|_{L^4} \|P_{\sim H_1}u_2\|_{L^4} + \|P_{\gtrsim H_1}u_1\|_{L^4}\|u_2\|_{L^4}\right) \|P_{\sim H}v\|_{L^2} \|P_Hv\|_{L^2}.
\end{align*}
Noticing that
$$
\sum_{H_1\in\D} H_1^{\frac 3{2\alpha}-\frac 34} \|P_{H_1}u_i\|_{L^\infty_TL^4_{xy}} \lesssim \sum_{H_1\in\D} H_1^{\frac 7{4\alpha}-\frac 58} \|P_{H_1}u_i\|_{L^\infty_TL^2_{xy}} \lesssim \|u_i\|_{B^{s_\alpha+}(T)},
$$
we deduce
\begin{equation}\label{ee.13}
  \sum_{H\in\D\setminus\{1\}} \sup_{t_H\in[0,T]} H^{2\sigma} \left|\int_0^{t_H} \mathcal{N}_H^1(v)dt\right| \lesssim \|u_1\|_{B^{s_\alpha+}(T)} \|u_2\|_{B^{s_\alpha+}(T)} \|v\|_{B^\sigma(T)}^2.
\end{equation}
Finally we evaluate the contribution of $\mathcal{N}_H^3(v)$ since by \eqref{symPi}, the term $\mathcal{N}_H^2(v)$ could be treated similarly. We perform a dyadic decomposition on $u$ and $v$ to obtain
\begin{align*}
  \mathcal{N}_H^3(v) &= c_1H^{-1} \int_{\R^2} \Pi_\eta(P_{\ll H}u, v) P_H\partial_x(P_{\ll H}uv) + c_1H^{-1}\int_{\R^2} \Pi_\eta(P_{\ll H}u,v)P_H\partial_x(uP_{\ll H}v)\\
  &\quad + c_1H^{-1} \sum_{H_1\gtrsim H} \int_{\R^2} \Pi_\eta(P_{\ll H}u,v) P_H\partial_x(P_{H_1}u P_{\sim H_1}v)\\
  &:= \mathcal{N}_H^{31}(v) + \mathcal{N}_H^{32}(v) + \mathcal{N}_H^{33}(v).
\end{align*}
By using estimate \eqref{PiEst} we infer that
\begin{align*}
  |\mathcal{N}_H^{31}(v)| &\lesssim \sum_{H_1,H_2\ll H} H^{\frac 1\alpha-1} H_1^{\frac 1{2\alpha}+\frac 14} \|P_{H_1}u\|_{L^2} \|P_{\sim H}v\|_{L^2} \|P_{H_2}u P_{\sim H}v\|_{L^2}\\
  &\lesssim \sum_{H_1,H_2\ll H} H_1^{\frac 1\alpha-\frac 14}\|P_{H_1}u\|_{L^2} H_2^{\frac 1\alpha-\frac 14}\|P_{H_2}u\|_{L^2} \|P_{\sim H}v\|_{L^2}^2,
\end{align*}
from which we deduce
\begin{equation}\label{ee.13}
  \sum_{H\in\D\setminus\{1\}} \sup_{t_H\in[0,T]} H^{2\sigma} \left|\int_0^{t_H} \mathcal{N}_H^{31}(v)dt\right| \lesssim \|u\|_{B^{s_\alpha+}(T)}^2  \|v\|_{B^\sigma(T)}^2.
\end{equation}
Then, observe that $\mathcal{N}_H^{32}(v) = \mathcal{N}_H^{31}(v)$ in the case $u=v$. Arguing as above we get for $u\neq v$ that
$$
|\mathcal{N}_H^{32}(v)|\lesssim \sum_{H_1,H_2\ll H} H^{\frac 1\alpha-1} H_1^{\frac 1{2\alpha}+\frac 14} H_2^{\frac 1{2\alpha}+\frac 14} \|P_{H_1}u\|_{L^2} \|P_{\sim H}u\|_{L^2} \|P_{\sim H}v\|_{L^2} \|P_{H_2}v\|_{L^2}.
$$
It follows that
\begin{equation}\label{ee.14}
  \sum_{H\in\D\setminus\{1\}} \sup_{t_H\in[0,T]} \left|\int_0^{t_H} \mathcal{N}_H^{32}(v)dt\right| \lesssim \|u\|_{B^{s_\alpha+}(T)}^2  \|v\|_{B^0(T)}^2,
\end{equation}
and at the $E^s$-level
\begin{equation}\label{ee.15}
  \sum_{H\in\D\setminus\{1\}} \sup_{t_H\in[0,T]} H^{2s} \left|\int_0^{t_H} \mathcal{N}_H^{32}(v)dt\right| \lesssim \|u\|_{B^{s}(T)}^2  \|v\|_{B^s(T)}^2.
\end{equation}
Finally we use similar arguments to bound $\mathcal{N}_H^{33}(v)$ and we obtain
\begin{equation}\label{ee.16}
  \sum_{H\in\D\setminus\{1\}} \sup_{t_H\in[0,T]} H^{2\sigma} \left|\int_0^{t_H} \mathcal{N}_H^{33}(v)dt\right| \lesssim \|u\|_{B^{s_\alpha+}(T)}^2  \|v\|_{B^\sigma(T)}^2.
\end{equation}
Gathering \eqref{ee.13}-\eqref{ee.16} we deduce
$$
  \sum_{H\in\D\setminus\{1\}} \sup_{t_H\in[0,T]} \left|\int_0^{t_H} \mathcal{N}_H(v)dt\right| \lesssim (\|u_1\|_{B^{s_\alpha+}(T)}\|u_2\|_{B^{s_\alpha+}(T)} + \|u\|_{B^{s_\alpha+}(T)}^2)  \|v\|_{B^0(T)}^2,
$$
and
$$
  \sum_{H\in\D\setminus\{1\}} \sup_{t_H\in[0,T]} H^{2s}\left|\int_0^{t_H} \mathcal{N}_H(v)dt\right| \lesssim (\|u_1\|_{B^{s_\alpha+}(T)}\|u_2\|_{B^{s_\alpha+}(T)} + \|u\|_{B^{s}(T)}^2)  \|v\|_{B^s(T)}^2,
$$
which combined with \eqref{ee.2}-\eqref{ee.3} and \eqref{ee.12} concludes the proof of Proposition \ref{ee}.
\end{proof}

\section{Proof of Theorem \ref{theo-main}. }\label{sec-main}

The proof of Theorem \ref{theo-main} closely follows the proof of existence and uniqueness given in \cite{KP}.  We start with a well-posedness result for smooth initial data $u_0$ in $E^{\infty}=H^{\infty}(\R^2)$. This result can be easily obtained with a parabolic regularization of (\ref{bozk}) by adding an extra term $-\varepsilon \Delta u$ and going to the limit as $\varepsilon \rightarrow 0$. We refer the reader to \cite{Io} for more details.

\begin{theorem}\label{theo-Hinfini}
  Assume that $u_0\in E^\infty$. Then there exist a positive time $T$ and a unique solution  $u\in C([-T,T];E^\infty)$ of (\ref{bozk}) with initial data $u(0,.)=u_0(.)$.
  Moreover $T=T(\|u_0\|_{E^3})$ is a nonincreasing function of $\|u_0\|_{E^3}$ and the flow-map is continuous.
\end{theorem}

\subsection{A priori estimates for $E^\infty$ solutions}
\begin{theorem}\label{est-theo-Hinfini}
  Assume that $s>s_\alpha$. For any $M>0$ there exists $T=T(M)>0$ such that, for all initial data $u_0\in E^\infty$ satisfying $\|u_0\|_{E^s}\leq M$, the smooth solution $u$ given by Theorem \ref{theo-Hinfini} is defined on $[-T,T]$ and moreover
  \begin{equation}\label{est1-Hinfini}
    u\in C([-T,T];E^{\infty}) \quad \mbox{and} \quad \|u\|_{L_T^{\infty} E^s}\lesssim \|u_0\|_{E^s}\;.
  \end{equation}
\end{theorem}
To obtain Theorem \ref{est-theo-Hinfini} we will need the following result proved in \cite{KP}.

\begin{lemma}\label{lemmaKPcontrole}
Assume that $s\geq 0$, $T>0$ and $u \in C([-T,T];E^\infty)$. Consider for $0\leq T'\leq T$
\begin{equation}\label{defLambdasT}
\Lambda ^s_{T'}(u)=\max{\left(\|u\|_{B^s_{T'}},\|\partial_x(u^2)\|_{N^s_{T'}} \right)}\;.
\end{equation}
The map $T'\mapsto \Lambda ^s_{T'}$ is nondecreasing,  continuous on $[0,T)$ and moreover
\begin{equation}\label{limLambdasT}
\lim_{T'\rightarrow 0}\Lambda ^s_{T'}(u)=0\;.
\end{equation}
\end{lemma}

{\it Proof of Theorem \ref{est-theo-Hinfini}}
First note that we can  always assume that the initial data $u_0$ have a small $E^s$-norm. Indeed, if $u(t,x,y)$ is a solution of (\ref{bozk}) then
$u_\lambda (t,x,y)=\lambda u(\lambda^{1+1/\alpha}t,\lambda^{1/\alpha}x,\lambda ^{1/2}y)$ is  a solution of (\ref{bozk}) on the time interval $[0,\lambda ^{-(1+1/\alpha)}T]$, with initial data
$u_\lambda (0,x,y)=\lambda u(\lambda^{1/\alpha}x,\lambda ^{1/2}y)$.
On the other hand, one can easily check that
\begin{equation}\label{dilatation}
\| u_{\lambda } (0,x,y) \|_{E^s} \lesssim \lambda ^{\frac{3}{4}-\frac{1}{2\alpha}} (1+\lambda ^s)\|u(0,x,y)\|_{E^s} \;,
\end{equation}
and then, choosing $\lambda \sim \varepsilon^{(\frac{3}{4}-\frac{1}{2\alpha})^{-1}}\|u_0\|_{E^s}^{(\frac{3}{4}-\frac{1}{2\alpha})^{-1}}$ we see that $u_{\lambda}(0,.)$ belongs to
$B^s(\varepsilon )$ the ball of $E^s$ centered at the origin with radius $\varepsilon$. Hence it is enough to prove that if $u_{\lambda} (0,.)\in B^s(\varepsilon)$,  Theorem \ref{est-theo-Hinfini} holds with $T=1$. This will prove the result with $T(\|u_0\|_{E^s})\sim \|u_0\|_{E^s}^{-(1+1/\alpha)(3/4-1/(2\alpha))}$.

In view of those considerations, we take now $u_0 \in E^{\infty}\cap B^s(\varepsilon )$ and let $u\in C([-T,T];E^{\infty})$ be the solution of (\ref{bozk}) given by Theorem \ref{theo-Hinfini} (with $0\leq T\leq 1$). Then gathering the linear estimate (\ref{FBN}), Proposition \ref{stbe}, (\ref{coer.0}) and (\ref{ee.000}) we get
\begin{equation}\label{emainth.0}
\Lambda^{\beta}_T(u)^2\lesssim (1+\|u_0\|_{E^0})\|u_0\|_{E^\beta}^2+(\Lambda ^s_T(u)+ \Lambda ^s_T(u)^2)\Lambda ^{\beta}_T(u)^2\;,
\end{equation}
for all $\beta\geq s >s_{\alpha}$.
Using (\ref{emainth.0}) with $\beta=s$, a continuity argument and that ${\displaystyle \lim_{t\rightarrow 0} \Lambda ^s_t(u)=0}$, we have
$\Lambda ^s_T(u)\lesssim \varepsilon$ as soon as $\|u_0\|_{E^s}\leq \varepsilon$. By estimate (\ref{FBN}) together with the short time estimate (\ref{stbe.1}) it follows then that for $\|u_0\|_{E^s}\leq \varepsilon$,
\begin{equation}\label{emainth.1}
\Gamma ^s_T(u)=\max{(\|u\|_{B^s_T},\|u\|_{F^s_T})}\lesssim \varepsilon\;.
\end{equation}

Then Lemma \ref{LinftyEF}, estimates (\ref{FBN}), (\ref{stbe.1}) and (\ref{emainth.0}) lead to
\begin{equation}\label{emainth.2}
\|u\|_{L^{\infty}_TE^{\beta}}\leq \Gamma ^{\beta}_T(u) \lesssim \|u_0\|_{E^{\beta}}\;,
\end{equation}
for all $\beta \geq s$ as soon as $\|u_0\|_{E^s}\leq \varepsilon$. Using this above estimate with $\beta=3$ we can apply Theorem \ref{theo-Hinfini} a finite number of time and thus extend the solution $u$ of (\ref{bozk}) on the time interval $[-1,1]$.

\subsection{$L^2$-Lipschitz bounds and uniqueness.}
Let us consider two solutions $u_1$ and $u_2$  defined on $[-T,T]$, with initial data $\varphi_1$ and $\varphi_2$ and assume moreover that
\begin{equation}\label{L2bound.0}
\varphi _i\in B^s(\varepsilon) \quad \mbox{and} \quad
\Gamma _T^{s_\alpha^+}(u_i)\leq \varepsilon \, ,\; i=1,2.
\end{equation}
If we define the function $v$ by $v=u_1-u_2$, we see that $v$ is a solution of (\ref{eq-u0}) with $u=u_1+u_2$ and moreover $u$ solves (\ref{eq-u1}) with a nonlinear term which is $u_1^2+u_2^2$. It follows then from (\ref{coer.0}), (\ref{ee.00}), (\ref{FBN}),  the short time estimate (\ref{stbe.01}) together with the smallness assumptions
(\ref{L2bound.0}) that
\begin{equation}\label{L2bound.1}
\Gamma ^0_T(v)\lesssim \|\varphi_1-\varphi_2\|_{L^2(\R^2)}\;.
\end{equation}
With this $L^2$-bound in hand we can now state our uniqueness result.
\begin{proposition}\label{prop-uniq}
Let $s>s_\alpha$. Consider $u_1$ and $u_2$ two solutions of (\ref{bozk}) in $C([-T,T];E^s) \cap B^s(T)\cap F^s(T)$ for some $T>0$. If $u_1(0,.)=u_2(0,.)$, then $u_1=u_2$ on the time interval $[-T,T]$.
\end{proposition}
{\it Proof.} Let be $C=\max{(\Gamma^s_T(u_1),\Gamma^s_T(u_2))}$. We consider the same dilatations $u_{i,\lambda}$ of $u_i$ as in the proof of Theorem \ref {est-theo-Hinfini}. As previously, they are solutions of (\ref{bozk}) on $[-T',T']$ with $T'=\lambda ^{-(1+1/\alpha)}T$ and with initial data $u_{i,\lambda}(0,x,y)=\lambda u(0,\lambda^{1\alpha}x,\lambda^{1/2}y)$. Then since we have
\begin{equation}\label{L2bound.2}
\|u_{i,\lambda}(0,.)\|_{E^s}\lesssim
\lambda^{3/4-1/(2\alpha)}(1+\lambda^s) \|u_{i,\lambda}(0,.)\|_{E^s} \;,
\end{equation}
and
\begin{align}\label{L2bound.3}
\|u_{i,\lambda}\|_{L^\infty_{T'}E^s}+\|u_{i,\lambda}\|_{B^s(T')}
 &\lesssim \lambda^{3/4-1/(2\alpha)}(1+\lambda^s)\left(\|u_{i,\lambda}\|_{L^\infty_{T}E^s}+\|u_{i,\lambda}\|_{B^s(T)}\right) \\
 &\lesssim  C\lambda^{3/4-1/(2\alpha)}(1+\lambda^s)\;.
\end{align}
Choosing $\lambda$ small enough we get
\begin{equation}\label{L2bound.2}
\|u_{i,\lambda}\|_{L^\infty_{T'}E^s}\lesssim \varepsilon \;,
\; \mbox{and} \;
\|u_{i,\lambda}\|_{B^s(T')}\lesssim \varepsilon\;.
\end{equation}
We prove now that for $\tilde{T}<T'$ small enough, we also have
\begin{equation}\label{l3bound.4}
\|u_{i,\lambda}\|_{F^s(\tilde{T})}\lesssim \varepsilon\;.
\end{equation}
Since $\|u_{i,\lambda}\|_{F^s(T)}\leq C$, we can always find $H\in \D$ such that
\begin{equation}\label{L2bound.5}
\|P_{>H}u_{i,\lambda}\|_{F^s(\tilde{T})} \leq \|P_{>H}u_{i,\lambda}\|_{F^s(T)}\leq \varepsilon\;, \, i=1,2.
\end{equation}
Moreover since $\|u\|_{\N^s(\tilde{T})}\leq C\|u\|_{L^2_{\tilde{T}}E^s}$, we infer from (\ref{FBN}), H\"older inequality and the Sobolev embedding $E^s\hookrightarrow H^{1/2}(\R^2)\hookrightarrow L^4(\R^2)$ (since $s\geq 1/2$) that
\begin{align}
\|P_{\leq H}u_{i,\lambda }\|_{F^s_{\tilde{T}}} & \leq  \|u_{i,\lambda }\|_{B^s_{\tilde{T}}}+\|P_{\leq H}\partial_x( u_{i,\lambda }^2)\|_{L^2_{\tilde{T}}E^s} \\
 &\leq  \|u_{i,\lambda }\|_{B^s_{\tilde{T}}}+\tilde{T}^{1/2}H^{s+1/\alpha}\|P_{\leq H} (u_{i,\lambda }^2)\|_{L^\infty_{\tilde{T}}L^2_{x,y}} \\
 &\leq  \varepsilon + \tilde{T}^{1/2}H^{s+1/\alpha}\| u_{i,\lambda }\|_{L^\infty_{\tilde{T}}L^4_{x,y}}^2 \\
 &\leq  \varepsilon + \tilde{T}^{1/2}H^{s+1/\alpha}\| u_{i,\lambda }\|_{L^\infty_{\tilde{T}}H^{1/2}}^2\;.
\end{align}
This leads to
\begin{align}\label{L2bound.6}
\|P_{\leq H}u_{i,\lambda }\|_{F^s_{\tilde{T}}} & \leq \varepsilon + \tilde{T}^{1/2}H^{s+1/\alpha}\| u_{i,\lambda }\|_{L^\infty_{\tilde{T}}H^{1/2}}^2 \\
&\leq \varepsilon + \tilde{T}^{1/2}H^{s+1/\alpha}\| u_{i,\lambda }\|_{L^\infty_{\tilde{T}}E^s}^2 \\
& \leq 2 \varepsilon\;,
\end{align}
by choosing $\tilde{T}$ small enough. Gathering estimates (\ref{L2bound.2}), (\ref{L2bound.5}) and (\ref{L2bound.6}), we thus obtain that the smallness condition (\ref{L2bound.0}) holds, which shows that $u_1=u_2$ on $[-\tilde{T},\tilde{T}]$ (since (\ref{L2bound.1}) holds). Using the same argument a finite number of time we obtain that $u_1=u_2$ on $[-T',T']$ and so on $[-T,T]$ by dilatation.

\subsection{Existence}

Let $s_\alpha <s<3$ and $u_0\in E^s$. By scaling considerations we can always assume that $u_0\in B^s(\varepsilon)$.
Following \cite{KP} we are going to use the Bona-Smith argument to obtain the existence of a solution $u$ with $u_0$ as initial data.

Consider $\rho\in S(\R^2)$
with $\int\rho (x,y)\,dxdy=1$ and $\int x^iy^j\rho (x,y)\, dxdy=0$ for $0\leq i\leq [s]+1$, $0\leq j\leq [s]+1$, $1\leq i+j$ and let us define $\rho_{\lambda}=\lambda^{1+1/\alpha }\rho(\lambda ^{1/\alpha}x, \lambda ^{1 /2}y)$. Then following \cite{BS} we have
\begin{lemma}\label{BonaSmith}
Let $s\geq 0$, $\varphi\in E^s$ and $\varphi_\lambda =\rho _\lambda * \varphi$. Then,
\begin{equation}\label{BS.0}
\| \varphi_\lambda \|_{E^{s+\delta }}\lesssim \lambda^{-\delta}\|\varphi\|_{E^s} \;,\, \forall \delta \leq 0 \;,
\end{equation}
and
\begin{equation}\label{BS.1}
\| \varphi_\lambda -\varphi \|_{E^{s-\delta }}=o\left( \lambda^{\delta}\right) \;,\, \forall  \delta \in {[0,s]}\;.
\end{equation}
\end{lemma}
Consider now the smooth initial data $u_{0,\lambda}=\rho _\lambda * u_0$. Since $u_{0,\lambda}\in H^\infty(\R^2)$ for any $\lambda >0$, by Theorem \ref{theo-Hinfini}, there exist $T_\lambda >0$ and an unique solution $u$ of (\ref{bozk}) such that $u_\lambda \in C([-T_\lambda , T_\lambda];H^\infty(\R^2))$ with initial data $u_\lambda(0,.)=u_{0,\lambda}$.
Note first that from (\ref{BS.0}) we have $\|u_{\lambda,0}\|_{E^s}\leq \|u_0\|_{E^s}\leq \varepsilon$. Hence following the proof of Theorem \ref{est-theo-Hinfini}, the sequence $(u_\lambda)$ can be extended on the interval $[-1,1]$ and moreover
\begin{equation}\label{exist.0}
\Gamma^s_1(u_\lambda)\leq C \|u_{\lambda,0}\|_{E^s}\lesssim \varepsilon \;\; \mbox{and} \; \;
\Gamma_1^{s+s_\alpha^+}(u_\lambda)\lesssim \|u_{0,\lambda}\|_{E^{s+s_\alpha^+}}\lesssim \lambda ^{-s_\alpha^+}\|u_0\|_{E^s}\;.
\end{equation}
Then we get from (\ref{L2bound.1}) and (\ref{BS.1}) that for $0<\lambda'\leq \lambda$,
\begin{equation}\label{exist.1}
\Gamma^0_1(u_\lambda - u_{\lambda '})\lesssim \|u_{0,\lambda}-u_{0,\lambda '}\|_{L^2(\R^2)}=o(\lambda ^{s})\;.
\end{equation}
Moreover, from estimates (\ref{FBN}), (\ref{stbe.1}), (\ref{ee.0}) we see that, for $s>s_\alpha$,
\begin{align}\label{exist.2}
\Gamma^s_1(u_\lambda - u_{\lambda '})^2 & \lesssim (1+  \|u_{0,\lambda}-u_{0,\lambda '}\|_{E^0} )\|u_{0,\lambda}-u_{0,\lambda '}\|_{E^s}^2\\
&  + (\Gamma^{s+s_\alpha ^+}_1(u_\lambda) + \Gamma^{s+s_\alpha ^+}_1(u_{\lambda '}) )\, \Gamma^0_1(u_\lambda - u_{\lambda '})\Gamma^s_1(u_\lambda - u_{\lambda '})\\
&  +(\Gamma_1^s(u_\lambda )^2+\Gamma_1^s(u_{\lambda '})^2)\Gamma_1^s(u_\lambda - u_{\lambda '})^2,
\end{align}
which leads to
\begin{equation}
\Gamma^s_1(u_\lambda - u_{\lambda '})^2\lesssim \|u_{0,\lambda}-u_{0,\lambda '}\|_{E^s}^2
+(\Gamma^{s+s_\alpha ^+}_1(u_\lambda) +\Gamma^{s+s_\alpha ^+}_1(u_{\lambda '}) )\, \Gamma^0_1(u_\lambda - u_{\lambda '})\,.
\end{equation}
Thus we have
\begin{equation}\label{exist.3}
\| u_\lambda - u_{\lambda '}\|_{L^\infty _1E^s}\lesssim  \Gamma^s_1(u_\lambda - u_{\lambda '})\rightarrow 0 \; \mbox{if} \; \lambda \rightarrow 0\;.
\end{equation}
This proves that the sequence $(u_\lambda)$ converges in the norm $\Gamma^s_1$ to a solution $u$ of (\ref{bozk}), which ends the proof.

\subsection{Continuity of the flow map} We refer to \cite{KP} for the continuity of the flow-map, which follows easily now from the results given in the previous subsections together with Theorem \ref{theo-Hinfini}

\section{Appendix.}
In this section we prove our $C^2$ ill-posedness result for initial data in $E^s$ (for all $s\in \R$) when $1\leq \alpha <2$. This extends previous results in \cite{EP} where the ill-posedness of (\ref{bozk}) is proved in $E^s$, for all $s\in \R$, assuming that $\alpha\leq 4/3$. This result has to be viewed as an extension of the well-known result in \cite{MST} where the $C^2$ ill-posedness in $H^s(\R)$ (for all $s\in \R$) of the one dimensional generalized Benjamin-Ono equation
$\partial_tu+D^\alpha_x u_x=uu_x$ is proved for all $\alpha \in [1,2[$.

Following \cite{MST}, we see that it is enough to build a sequence of functions $f_N$ such that, for all $s\in \R$,
\begin{equation}\label{CS1 fn}
\|f_N\|_{E^s} \leq C \;,
\end{equation}
and
\begin{equation} \label{CS2 fn}
\lim_{N\rightarrow +\infty}\left\| \int_0^t U(t-t')[U(t')f_NU(t')(f_N)_x]\, dt' \right\|_{E_s}=+\infty \,.
\end{equation}
Let $N$ large enough, $\gamma \ll 1$ and $0<\varepsilon \ll 1$ such that $\gamma ^\varepsilon \lesssim 1$. Let us now define the subsets of $\R^2$,
\begin{equation}\label{defQi}
 Q^+_1=[\gamma/2,\gamma]\times [\gamma ^{\varepsilon},2\gamma^{\varepsilon}]
 \quad \mbox{and} \quad
Q^+_2=[N,N+\gamma]\times [-\gamma ^{\varepsilon}/2,-\gamma^{\varepsilon}] \,.
\end{equation}
Then define $Q^-_1=-Q^+_1$ and $Q^-_2=-Q^+_2$. We consider $f_N$ defined through its Fourier transform by
 \begin{equation}\label{def fn}
 \F (f_N)(\zeta )=\gamma^{-\frac{1+\varepsilon}{2}}\left( 1_{Q^+_1}(\zeta )+1_{Q^-_1}(\zeta )\right)
 +\gamma^{-\frac{1+\varepsilon}{2}}N^{-\alpha s}\left( 1_{Q^+_2}(\zeta )+1_{Q^-_2}(\zeta )\right) \;.
 \end{equation}
 Clearly the sequence $f_N$ is real valued and moreover (\ref{CS1 fn}) holds by obvious calculations.
 Consider now
 $$
 I_N(t,x,y)=\int_0^t U(t-t')[U(t')f_NU(t')(f_N)_x]\, dt'.
 $$
 Standard calculations leads then to
 $$
 I_N=\int_{\R^4} e^{i(x\xi+y\mu+t\omega (\zeta))}\, \xi \, \F (f_N)(\zeta _1)\F (f_NJ)(\zeta - \zeta _1)
 \frac{e^{it\Omega (\zeta_1,\zeta - \zeta _1)}-1}{\Omega (\zeta_1 , \zeta -\zeta_1)}\,d\zeta d\zeta_1\;,
 $$
  with $\Omega (\zeta _1 ,\zeta -\zeta _1)=\omega (\zeta)-\omega(\zeta _1)-\omega (\zeta - \zeta _1)$.
  By localization considerations, observe now that $I_N$ can be rewritten as the sum of eight terms with disjoint supports corresponding to each  different interactions in the nonlinear term.
  Hence, considering only the low-high interaction $1_{Q^+_1}(\zeta )1_{Q^+_2}(\zeta)$,  it will be enough to prove that (\ref{CS2 fn}) holds where $\F (I_N)$ is now replaced by
  $$
  \F(I_N)(t,\zeta )=\gamma ^{-(1+\varepsilon )}N^{-\alpha s}  e^{it\omega (\zeta)}\xi
  \int_{\zeta _1\in Q^+_1,\zeta - \zeta _1\in Q^+_2}
 \frac{e^{it\Omega (\zeta_1, \zeta -\zeta _1)}-1}{\Omega (\zeta_1, \zeta -\zeta_1)}\, d\zeta_1\;.
  $$
  We claim now that for $\zeta _1\in Q_1^+$, $\zeta - \zeta _1\in Q^+_2$ and $\gamma =o(N)$, then it holds,
  \begin{equation}\label{min phase}
 | \Omega (\zeta _1,\zeta - \zeta _1)| \sim \gamma N^\alpha \;.
  \end{equation}
  Recall first that
  $$
  \Omega (\zeta _1,\zeta - \zeta _1)=[\xi |\xi |^\alpha - \xi _1|\xi _1|^\alpha - (\xi - \xi _1) |\xi - \xi _1 |^\alpha ]
  +[\xi\mu ^2 -\xi _1\mu _1^2- (\xi - \xi _1)(\mu -\mu _1)^2 ]=I+II \;.
  $$
  {\it {\bf . Contribution I}}

 By virtue of the mean value theorem we infer that there exists $\theta \in [\xi - \xi_1,\xi]$ such that
 $$
 \left| \xi |\xi |^\alpha - (\xi - \xi _1) |\xi - \xi _1 |^\alpha \right|=(\alpha +1)| \xi _1 | |\theta |^{\alpha}\;,
 $$
 which leads to
  \begin{equation}\label{est I1}
   | \xi |\xi |^\alpha - (\xi - \xi _1) |\xi - \xi _1 |^\alpha | \sim \gamma N^\alpha \;,
  \end{equation}
Moreover, recalling that $| \xi _1 | \sim \gamma =o(N)$ we have
  \begin{equation}\label{est I2}
  | \xi _1|\xi _1|^\alpha | \sim \gamma ^{\alpha +1}=o(N) \;.
  \end{equation}
  Then gathering (\ref{est I1}) and ({\ref{est I2}) we obtain
  \begin{equation}\label{est I}
  I\simeq \gamma N^\alpha \;.
  \end{equation}
  {\it {\bf . Contribution II}}

   Since $\zeta _1\in Q_1^+$ and  $\zeta - \zeta _1\in Q^+_2$, then
   ${\displaystyle
   N+\frac{\gamma}{2} \leq \xi \leq N+2\gamma }$ and
  ${\displaystyle
   \frac{1}{2}\gamma^{\varepsilon} \leq \mu \leq \gamma ^{\varepsilon}
   }$
   which leads to
   \begin{equation}\label{est II1}
   \frac{1}{4}\gamma ^{2\varepsilon }\left(N+\frac{\gamma}{2}\right)
   \leq \xi \mu^2 \leq
   \gamma ^{2\varepsilon }\left( N+ 2\gamma \right) \;.
   \end{equation}
  On the other hand, since  $\zeta - \zeta _1\in Q^+_2$ we have

\begin{equation}\label{est II2}
   -\frac{1}{4}\gamma ^{2\varepsilon }\left(N+\gamma \right)
   \leq  - (\xi - \xi _1)( \mu - \mu_1)^2 \leq
   -\gamma ^{2\varepsilon }N\;.
   \end{equation}
   In the same way, since $\zeta _1\in Q_1^+$ we have
   \begin{equation}\label{est II3}
  | \xi _1 \mu_1 ^2 | \sim \gamma ^{1+2\varepsilon }
   \end{equation}
   Then gathering (\ref{est II1}),(\ref{est II2}) and (\ref{est II3}) we infer that

\begin{equation}\label{est II}
   II = O(\gamma ^{1+2\varepsilon }) \;.
   \end{equation}
   Then (\ref{min phase}) follows from (\ref{est I}) together with (\ref{est II}).

   Choosing now $\gamma =N^{-(\alpha +\delta)}$ for some $\delta >0$, it follows  from (\ref{min phase}) that
   $$
   \left| \frac{e^{it\Omega (\zeta_1, \zeta - \zeta _1)}-1}{\Omega (\zeta_1, \zeta -\zeta_1)} \right| \sim |t|,
   $$
   which lead then to
   \begin{equation}\label{est In}
   \|I_N\|_{E^s}^2 \gtrsim
   \gamma ^{-2(1+\varepsilon)} N^{-2\alpha s+2 }|t|\gamma^{1+\varepsilon}
    (N^\alpha + \gamma ^{2\varepsilon})^{2s}\gamma ^{2(1+\varepsilon)} .
   \end{equation}
   Thus, choosing $\varepsilon (\alpha)$ and $\delta (\alpha)$ small enough, we have
   $$
   \lim_{N\rightarrow + \infty}\|I_N\|_{E^2}^2 \gtrsim
   \lim_{N\rightarrow + \infty} |t |N^2\gamma^{1+\varepsilon}\gtrsim
   \lim_{N\rightarrow + \infty}|t| N^{2-\alpha -\varepsilon (\alpha +\delta )-\delta }
   = +\infty \;,
   $$
   for all $\alpha \in {[1,2[}$ and for all $s\in \R $. This ends the proof of (\ref{CS2 fn}).

\end{document}